\documentclass[sn-mathphys-num,pdflatex]{sn-jnl}

\usepackage{fix-cm}
\usepackage{graphicx}
\usepackage{amsmath,amssymb,amsfonts}
\usepackage{amsthm}
\usepackage[title]{appendix}
\usepackage{booktabs}
\usepackage{mathtools}
\usepackage{enumitem}
\usepackage{longtable}
\usepackage{array}
\usepackage{tabularx}
\usepackage{tikz}

\hypersetup{
  colorlinks=true,
  linkcolor={blue!80!black},
  citecolor={red!80!black},
  urlcolor={blue!80!black},
  bookmarksdepth=subsubsection,
  pdftitle={Independent Sets in Multiset Profile Graphs via Weighted Local Covers},
  pdfauthor={Aryeh Lev Zabokritskiy (Yohananov)},
  pdfsubject={Combinatorics; independent sets; discrete simplices},
  pdfkeywords={multiset profiles, spreading numbers, prime-checksum conjecture, independent sets, weighted local covers, discrete simplices, multiset codes}
}

\allowdisplaybreaks
\renewcommand{\le}{\leqslant}
\renewcommand{\ge}{\geqslant}

\DeclareMathOperator{\cost}{cost}
\DeclareMathOperator{\dens}{dens}
\newcommand{\tileind}{\mathbf{1}}

\theoremstyle{plain}
\newtheorem{theorem}{Theorem}[section]
\newtheorem{proposition}[theorem]{Proposition}
\newtheorem{corollary}[theorem]{Corollary}
\newtheorem{lemma}[theorem]{Lemma}
\newtheorem{conjecture}[theorem]{Conjecture}

\theoremstyle{definition}
\newtheorem{definition}[theorem]{Definition}

\theoremstyle{remark}
\newtheorem{remark}[theorem]{Remark}

\begin{document}

\title[Independent Sets via Weighted Local Covers]
{Independent Sets in Multiset Profile Graphs via Weighted Local Covers}

\author[1]{\fnm{Aryeh~Lev} \sur{Zabokritskiy (Yohananov)}}
\email{yuhanalev@telhai.ac.il}

\affil[1]{%
\orgdiv{Department of Computer Science},
\orgname{MIGAL -- Galilee Research Institute/Tel-Hai University of Kiryat Shmona and the Galilee},
\orgaddress{\street{P.O. Box 831}, \city{Kiryat Shmona}, \postcode{1101602}, \country{Israel}}%
}

\abstract{
Let $G_q(d)$ be the unit-transfer graph on the nonnegative integer vectors
whose $q$ coordinates sum to $d$, equivalently on the multiplicity profiles
of size-$d$ multisets over $q$ symbols.  The prime-checksum conjecture predicts
that, for prime $q$ and all sufficiently large $d$, a largest independent set
is a fiber of the natural cyclic checksum.  We develop a finite-state
weighted local-cover method for $G_q(d)$: translated induced subgraphs give
local independence inequalities, while capped anchor profiles reduce the
covering conditions for infinitely many degrees to a finite rational linear
system.

This method gives new proofs of the known cases $q=3$ and $q=4$ and determines
$\alpha(G_q(d))$ exactly for $q=5$ and $q=7$ in every degree, thereby proving
the next two odd-prime cases of the conjecture.  In the complementary regime
where $d$ is fixed and $q$ grows, the same method gives an explicit upper bound
for $\alpha(G_q(5))$ for every $q\ge7$, determines $\alpha(G_q(d))$ exactly
when $q$ is a power of two and $d\in\{6,8,10\}$, and yields an asymptotically
sharp upper bound through three terms for every fixed $d\ge7$.  The finite
systems arising in these arguments are verified in exact arithmetic and
supported by independently checkable certificates.
}



\maketitle
\section{Introduction}
\label{sec:intro}

The spreading number $\alpha_q(d)$ is the largest size of a family of
size-$d$ multiset profiles over $q$ symbols in which no two profiles differ by
replacing one occurrence of one symbol by another.  Equivalently, it is the
independence number of the unit-transfer graph on weak compositions of $d$.
The natural cyclic coloring gives, for prime $q$, an independent checksum
fiber of size
\[
 \left\lceil\frac1q\binom{d+q-1}{q-1}\right\rceil.
\]
The prime-checksum conjecture asserts that this construction is optimal for
all sufficiently large $d$ and is the central problem of the paper.

\begin{conjecture}[Prime-alphabet checksum optimality]
\label{conj:prime-checksum-optimality}
For every prime $q$, there is an integer $d_0(q)$ such that
\[
 \alpha_q(d)=
 \left\lceil\frac1q\binom{d+q-1}{q-1}\right\rceil
 \qquad(d\ge d_0(q)).
\]
\end{conjecture}

Under the correspondence
$a\leftrightarrow x_1^{a_1}\cdots x_q^{a_q}$, the same graph is the graph on
degree-$d$ monomials studied by Geramita, Gregory, and Roberts, who determined
the classical cases $q=3$ and $q=4$~\cite{GeramitaGregoryRoberts1986}.
Subsequent work developed the spreading and covering parameters further
\cite{CarliniHaVanTuyl2001,BabcockVanTuyl2013}.  Machacek later studied
uniqueness and rigidity of maximum independent sets
\cite{Machacek2021}.  These cases have special low-dimensional structure.  In
three variables the simplex is planar and its boundary consists of
one-dimensional paths; in four variables a parity decomposition remains
available.  The known arguments exploit precisely these features.

Our earlier joint work analyzed the cyclic checksum construction in the
multiset-deletion setting, proved its asymptotic optimality for each fixed
alphabet, obtained the exact cases $q=3$ and $q=4$, and formulated the
prime-alphabet conjecture~\cite{KreindelEssayagZabokritskiy2026a}.  It did not
settle $q=5$ or $q=7$, and it did not contain the finite-state local-cover
method developed here.

The cases $q=5$ and $q=7$ are harder because their profile simplices have
dimensions four and six.  The recurrences used in dimensions two and three no
longer control the boundary.  Our method covers the graph by translated
smaller graphs with known independence numbers.  The weight of a translate
depends only on a capped version of its anchor profile.  This leaves a finite
system of linear inequalities, even though the result applies to every large
$d$.

Our principal results prove the prime-checksum conjecture for the next two
prime alphabets.  For both $q=5$ and $q=7$ we obtain
\begin{equation}
 \alpha_q(d)=\left\lceil\frac1q\binom{d+q-1}{q-1}\right\rceil
 \label{eq:intro-q5-q7-main}
\end{equation}
apart from a few low degrees, whose exact values are determined in the
corresponding sections.  Thus both spreading numbers are completely
determined, and the first unresolved odd-prime case is $q=11$.  The same
method also gives new unified graph-theoretic proofs of the known three- and
four-variable results.

\begin{table}[htbp]
\centering
\caption{The prime cases determined in this paper.}
\label{tab:prime-checksum-status}
\small
\begin{tabularx}{\linewidth}{@{}c >{\raggedright\arraybackslash}X
 >{\raggedright\arraybackslash}X@{}}
\toprule
$q$&degrees in which the checksum formula holds&exceptions\\
\midrule
3&all $d\notin\{2,4\}$&$\alpha_3(2)=3$, $\alpha_3(4)=6$\\
5&all $d\notin\{2,4\}$&$\alpha_5(2)=5$, $\alpha_5(4)=16$\\
7&all $d\notin\{2,3,4,6\}$&$\alpha_7(2)=7$, $\alpha_7(3)=14$,
$\alpha_7(4)=35$, $\alpha_7(6)=133$\\
\bottomrule
\end{tabularx}
\end{table}

The upper-bound method is a weighted double-counting argument.  We place
induced subgraphs with known independence numbers throughout $G_q(d)$ and
assign them nonnegative weights so that every vertex is covered with total
weight at least one.  An independent set meets each placed subgraph in at most
its local independence number; summing these local inequalities therefore
gives a global upper bound.  For fixed $q$, we use structured translates whose
weights depend only on capped anchor profiles, so the boundary configurations
arising in infinitely many degrees are represented by a finite system of
inequalities.  For fixed $d$ and growing $q$, grouping profiles and anchors by
integer-partition type gives an analogous orbit system.  Computation enters
only after these mathematical reductions have been proved.  The finite systems
are checked in exact arithmetic rather than numerically; the remaining integer
or Boolean claims are verified by deterministic enumeration or by RUP/DRAT
derivations checked independently of the search program.
Section~\ref{subsec:certificate-verification} describes the common verification
procedure, and the exact certificates and checking tools are archived in the
computational companion~\cite{ZabokritskiyCompanion2026}.
Through the multiplicity-vector correspondence, these independent-set
theorems also determine optimal one-deletion-correcting multiset codes for
alphabet sizes five and seven, while recovering the known sizes for three
and four symbols; see~\cite{KreindelEssayagZabokritskiy2026a} for the broader
coding framework.
When $d$ is fixed and $q$ grows, profiles fall into finitely many partition
types.  Grouping them by type gives, for every $q\ge7$, the explicit upper bound
\[
 \alpha_q(5)\le
 \left\lfloor\frac{q(q^3+10q^2+45q+64)}{120}\right\rfloor.
\]
The same approach determines $\alpha_q(d)$ exactly for every power-of-two value
of $q$ when $d\in\{6,8,10\}$, and gives an asymptotically sharp upper bound
with three explicit terms for every fixed $d\ge7$.

\paragraph{Organization of the paper.}
Section~\ref{sec:pre} defines the discrete simplex, the unit-transfer graph,
the additive colorings that provide independent sets, and the translated
local tiles used for upper bounds.  Section~\ref{sec:tilings} proves the
weighted local-cover inequality and the finite-state reduction: for a fixed
tile family, infinitely many vertex-covering conditions are reduced to one
finite rational linear system using capped anchor profiles.
Sections~\ref{sec:q3frac} and~\ref{sec:q4frac} apply the method to the known
three- and four-variable cases.  The first uses translated triangular
templates and a planar boundary correction; the second separates odd and even
degrees and combines parity with exact-density tiles.  These sections provide
new proofs of the earlier formulas and illustrate the two basic forms of the
method.  Section~\ref{sec:q5} treats five variables.  Its eventual construction uses
only translated full residual simplices and gives a particularly symmetric
certificate for every $d\ge35$; six smaller exact-density templates and a
certified RUP refutation settle the finite range.  Section~\ref{sec:q7}
introduces orbit-union boundary tiles and completes the exact seven-variable
result by exact rational and DRAT certificates for the remaining degrees.
Section~\ref{sec:general-framework} determines the sizes of the additive color
classes for cyclic and more general finite abelian label groups.  It also
states the proposed extension of the finite-tile method to further fixed
alphabets.  Section~\ref{sec:fixed-degree} turns to the opposite regime, in
which $d$ is fixed and $q$ grows.  It derives the partition-orbit linear
program, obtains the explicit degree-five upper bound stated above, proves the
power-of-two families in degrees six, eight, and ten, records the point at
which the simpler clique family stops working, and proves the improved
three-term asymptotic upper bound for every fixed $d\ge7$.  Finally,
Section~\ref{sec:coding} gives the multiset-deletion
interpretation, summarizes the main results, and lists the remaining open
problems.

\section{Multiset profile graphs, additive colorings, and local templates}
\label{sec:pre}

Throughout, \(q\ge2\) is the number of symbols and \(d\ge0\) is the multiset
size.  We write \([q]=\{1,\ldots,q\}\), and let \(e_1,\ldots,e_q\) be the
standard unit vectors of \(\mathbb Z^q\).

\begin{definition}[Profile graph]
\label{def:monomial-graph}
A size-$d$ multiset profile over $q$ symbols is a vector in
\[
 \Delta_q(d)=\left\{a\in\mathbb Z_{\ge0}^q:\sum_{i=1}^q a_i=d\right\},
 \qquad
 |\Delta_q(d)|=\binom{d+q-1}{q-1}.
\]
The profile graph $G_q(d)$ has vertex set $\Delta_q(d)$; distinct profiles
\(a,b\) are adjacent exactly when \(a-b=e_i-e_j\) for some \(i\ne j\),
equivalently when
\(\lVert a-b\rVert_1=2\).  We write
\[
 \alpha_q(d)=\alpha(G_q(d)).
\]
\end{definition}

The map $a\mapsto x_1^{a_1}\cdots x_q^{a_q}$ identifies $G_q(d)$ with the
degree-$d$ monomial graph from the algebraic literature.  We use this
interpretation only when comparing with earlier work; all arguments below are
stated in terms of profiles.

Figure~\ref{fig:g3-3-monomials} shows this correspondence in the smallest
triangular example.  The three corners are the pure powers, and moving along
an edge transfers one unit of exponent from one variable to another.

\begin{figure}
\centering
\begin{tikzpicture}[scale=1.08]
\def\h{0.92}
\coordinate (a) at (0,{3*\h});
\coordinate (b) at (-0.5,{2*\h});
\coordinate (c) at ( 0.5,{2*\h});
\coordinate (d) at (-1,\h);
\coordinate (e) at ( 0,\h);
\coordinate (f) at ( 1,\h);
\coordinate (g) at (-1.5,0);
\coordinate (h) at (-0.5,0);
\coordinate (i) at ( 0.5,0);
\coordinate (j) at ( 1.5,0);
\draw (a)--(b) (a)--(c);
\draw (b)--(c) (b)--(d) (b)--(e);
\draw (c)--(e) (c)--(f);
\draw (d)--(e) (e)--(f);
\draw (d)--(g) (d)--(h);
\draw (e)--(h) (e)--(i);
\draw (f)--(i) (f)--(j);
\draw (g)--(h) (h)--(i) (i)--(j);
\node[fill=white,inner sep=1.5pt] at (a) {$x_1^3$};
\node[fill=white,inner sep=1.5pt] at (b) {$x_1^2x_2$};
\node[fill=white,inner sep=1.5pt] at (c) {$x_1^2x_3$};
\node[fill=white,inner sep=1.5pt] at (d) {$x_1x_2^2$};
\node[fill=white,inner sep=1.5pt] at (e) {$x_1x_2x_3$};
\node[fill=white,inner sep=1.5pt] at (f) {$x_1x_3^2$};
\node[fill=white,inner sep=1.5pt] at (g) {$x_2^3$};
\node[fill=white,inner sep=1.5pt] at (h) {$x_2^2x_3$};
\node[fill=white,inner sep=1.5pt] at (i) {$x_2x_3^2$};
\node[fill=white,inner sep=1.5pt] at (j) {$x_3^3$};
\end{tikzpicture}
\caption{The profile simplex $G_3(3)$ with vertices written as degree-three
monomials; compare~\cite[Figure~1]{Machacek2021}.}
\label{fig:g3-3-monomials}
\end{figure}
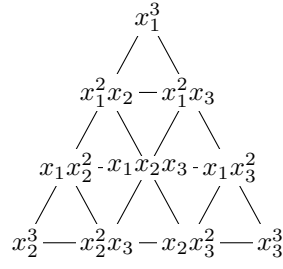

For the basic lower bound, let \(A\) be an abelian group of order \(q\) and
label the coordinates by its distinct elements \(g_1,\ldots,g_q\).  Color a
profile by
\begin{equation}
 \chi(a)=\sum_{i=1}^q a_i g_i,
 \qquad a\in\Delta_q(d).
 \label{eq:additive-coloring}
\end{equation}
For \(s\in A\), let
\[
 I_s(d)=\{a\in\Delta_q(d):\chi(a)=s\}.
\]
The cyclic specialization takes \(A=\mathbb Z_q\) and \(g_i=i-1\).  We write
\begin{equation}
 \sigma_q(a)=\sum_{i=1}^q(i-1)a_i\pmod q
 \label{eq:cyclic-checksum}
\end{equation}
and denote its fibers by
\begin{equation}
 \mathcal C_{q,d}(r)=\{a\in\Delta_q(d):\sigma_q(a)=r\},
 \qquad
 N_r^{(q)}(d)=|\mathcal C_{q,d}(r)|.
 \label{eq:checksum-classes}
\end{equation}

If \(a-b=e_i-e_j\) with \(i\ne j\), then
\(\chi(a)-\chi(b)=g_i-g_j\ne0\).  Thus \(\chi\) is a proper \(q\)-coloring and
each fiber is independent.  Since the fibers partition \(\Delta_q(d)\),
\begin{equation}
 \alpha_q(d)\ge
 \left\lceil\frac1q\binom{d+q-1}{q-1}\right\rceil.
 \label{eq:additive-lower-bound}
\end{equation}

For three variables these fibers give the familiar periodic pattern in the
triangular grid.  The classical proofs for three and four variables and the
uniqueness results of Machacek show that the decisive issue is how this
interior pattern meets the boundary~\cite{GeramitaGregoryRoberts1986,Machacek2021}.
Translated induced subgraphs are particularly well suited to that boundary
problem.  Their local independence numbers are unchanged, while their
incidences are governed by exponent arithmetic.  These translations take place
in the semigroup \(\mathbb Z_{\ge0}^q\), not in a quotient group.
Let \(H=G_q(r)[U]\) be an induced subgraph of \(G_q(r)\), with
\(U\subseteq\Delta_q(r)\), and let \(b\in\Delta_q(d-r)\).  The translate of
\(H\) by \(b\) is
\[
 b+H:=G_q(d)[b+U],
 \qquad b+U=\{b+h:h\in U\}.
\]
When used in a covering argument, this translated copy is a \emph{local tile}.
The purpose of these tiles is to turn a known small independence number into
a local upper bound inside the large graph.
The map \(h\mapsto b+h\) is a graph isomorphism because translation preserves
differences of exponent vectors.  Hence
\begin{equation}
 \alpha(b+H)=\alpha(H).
 \label{eq:translate-preserves-alpha}
\end{equation}
For every independent set \(I\subseteq G_q(d)\),
\begin{equation}
 |I\cap V(b+H)|\le\alpha(H).
 \label{eq:local-tile-bound}
\end{equation}
The next section combines many inequalities of this form.  If their weighted
sum counts every ambient vertex at least once, it gives an upper bound on the
size of every independent set and hence on $\alpha_q(d)$.

Three templates recur below.  The upward simplex
\(b+\Delta_q(1)=\{b+e_i:i\in[q]\}\) is a \(K_q\); a translated full simplex
\(b+\Delta_q(r)\) is a copy of \(G_q(r)\); and reflected sets of the form
\(b+\{\mathbf1-e_i:i\in[q]\}\) are again \(q\)-cliques.  Theorem~\ref{thm:finite-state-reduction} applies to every finite family of
such shapes.  The next sections use full simplices, the two clique
orientations, and a small number of boundary-correcting templates.

\section{Finite-state weighted local decompositions}
\label{sec:tilings}

The small graph and its position in the ambient simplex play different roles.
A base tile is a fixed finite graph $H=G_q(r)[U]$ inside the small simplex
$\Delta_q(r)$.  For every anchor $a\in\Delta_q(d-r)$, the translate
$a+H=G_q(d)[a+U]$ places the same shape at a new position in the larger
simplex.  Thus a finite tile is moved throughout the ambient space by varying
$a$; a weighted tiling assigns weights to these translated copies, which may
overlap and need not form a partition.  Choosing the base tiles is not
automatic.  Full simplices suffice in the eventual three- and five-variable arguments,
the four-variable argument also needs the two clique orientations, and the
seven-variable argument requires mixed profile layers near the boundary.  The
finite-state theorem does not discover these shapes: once they are chosen, it
reduces the search for weights to a finite system, and the later sections show
that the stated choices give the required bounds.  To separate this universal
covering argument from the particular geometry, let $G=(V,E)$ be a finite
graph and let $\mathcal T$ be a finite indexed family of induced subgraphs of
$G$.  Its members are the placed tiles; they may overlap, need not be cliques,
and need not be mutually isomorphic.  No algebraic condition is imposed: for
every subset $W\subseteq V$, the induced graph $G[W]$ is an admissible tile.
In particular, a tile need not be a coset, a simplex, an orbit, or a member of
a partition.  In the profile-graph applications we choose translates of
algebraically structured subsets because their independence numbers and
incidences can be controlled; this is a feature of our constructions, not a
hypothesis of the local-cover bound.
Assign a nonnegative weight $w_T$ to every $T\in\mathcal T$.  The coverage of
a vertex and the independence cost are
\[
 F(x)=\sum_{T\in\mathcal T}w_T\mathbf 1_{V(T)}(x),
 \qquad
 \cost(\mathcal T,w)=\sum_{T\in\mathcal T}w_T\alpha(T).
\]
Here $F(x)$ is simply the total weight of the placed tiles that contain $x$.
It is the bookkeeping tool that turns the local inequalities into the desired
global upper bound: the condition $F(x)\ge1$ ensures that every vertex of an
independent set is counted with total weight at least one, while
$\cost(\mathcal T,w)$ is the resulting upper-bound cost.  The weighted family
is a \emph{fractional local cover} when $F(x)\ge1$ for all $x\in V$.  Thus a
correct choice of tile shapes and weights that gives $F\ge1$ at the target
cost is exactly a solution of the upper-bound problem; the fixed-alphabet
sections construct such choices.

\begin{theorem}[Weighted local-cover bound]
\label{thm:weighted-tile-bound}
Every fractional local cover of $G$ satisfies
\[
 \alpha(G)\le \sum_{T\in\mathcal T}w_T\alpha(T).
\]
For a fixed template family, the best bound of this form is the value of
\begin{equation}
 \begin{aligned}
 \text{minimize}\quad&\sum_{T\in\mathcal T}\alpha(T)w_T,\\
 \text{subject to}\quad&
       \sum_{\substack{T\in\mathcal T\\x\in V(T)}}w_T\ge1
       &&(x\in V),\\
 &w_T\ge0&&(T\in\mathcal T).
 \end{aligned}
 \label{eq:abstract-tile-lp}
\end{equation}
\end{theorem}

\begin{proof}
For every independent set $I$,
\[
 |I|\le\sum_{x\in I}F(x)
 =\sum_{T\in\mathcal T}w_T|I\cap V(T)|
 \le\sum_{T\in\mathcal T}w_T\alpha(T).
\]
Maximizing over $I$ proves the assertion.
\end{proof}

\begin{remark}
The translated or coset-like tiles used later are a feature of our
constructions, not a requirement of the covering argument.  The proof above
uses only that each tile is an induced subgraph, and therefore applies
verbatim to an arbitrary family of vertex subsets.  Consequently, the two
elementary examples below may legitimately use cliques and a singleton that
do not arise as cosets or as translates of a base tile.
\end{remark}

As a concrete first example, let
$G_0=K_3^{(1)}\sqcup K_3^{(2)}\sqcup K_3^{(3)}$.  Use the three components as
tiles and give each weight one.  Then $F\equiv1$, the cost is three, and
choosing one vertex from each component gives
\[
 \alpha(G_0)=3=\cost(\mathcal T,w).
\]
The same argument shows that for
$K_{n_1}\sqcup\cdots\sqcup K_{n_m}$ the component cliques with unit weights
give the optimal cost $m$ immediately.

Figure~\ref{fig:coverage-intuition} shows why both the tile family and its
weights matter once the tiles interact.  Let $L=\{a_1,a_2,a_3\}$ and
$R=\{b_1,b_2,b_3\}$ be disjoint triangles, and let $x$ be adjacent to all six
side vertices, with no edges between $L$ and $R$.  Then $\alpha(G)=2$: one may
choose one vertex from each side, whereas choosing $x$ excludes every other
vertex.  In panel (a), assign weights $w_L,w_x,w_R$ to the disjoint tiles
$L,\{x\},R$.  The coverage constraints force
$w_L\ge1$, $w_x\ge1$, and $w_R\ge1$, so every feasible choice has cost at
least three.  Unit weights attain this minimum and give $F\equiv1$, but the
resulting bound is not sharp.  No different choice of weights can repair this
tile family.

In panel (b), take instead the two overlapping tiles
$T_L=G[L\cup\{x\}]$ and $T_R=G[R\cup\{x\}]$, both isomorphic to $K_4$, with
weights $s$ and $t$.  Their coverages and cost are
\[
 F(a_i)=s,\qquad F(b_j)=t,\qquad F(x)=s+t,
 \qquad \cost=s+t.
\]
Indeed, each $a_i$ lies only in $T_L$, each $b_j$ lies only in $T_R$, and $x$
lies in both tiles.  Covering the left and right side vertices therefore forces
$s\ge1$ and $t\ge1$; covering $x$ imposes only the weaker condition
$s+t\ge1$.  The minimum is attained at $s=t=1$.  At these weights every side
vertex has coverage one and $x$ has coverage two.  Since each tile is a
$K_4$, it has independence number one, and the weighted local-cover bound gives
\[
 \alpha(G)\le 1\cdot\alpha(T_L)+1\cdot\alpha(T_R)=2.
\]
Conversely, any pair consisting of one vertex of $L$ and one vertex of $R$ is
independent, because there are no edges between the two sides.  Hence
$\alpha(G)=2$, and the bound from panel~(b) is sharp.  By contrast,
$s=t=1/2$ covers the central vertex exactly but leaves every side vertex with
coverage $1/2$, so it is not feasible.  A successful certificate therefore
requires both suitable tile shapes and suitable weights.

The disjoint situation in panel~(a) is a useful baseline, but it is less
representative of our applications.  There the tiles are algebraically
structured translates, often of coset-like sets, and different translates
typically overlap; the relevant difficulty is therefore closer to
panel~(b).  The algebraic structure lets us control the independence number
of each tile and the pattern of its incidences, but it does not by itself make
the resulting bound sharp.  The main work of the paper is to choose structured
tile families and weights for which these overlapping local bounds combine to
give the desired exact global bound.

\begin{figure}
\centering
\begin{tikzpicture}[
  scale=.73,transform shape,
  vertex/.style={circle,draw,fill=white,minimum size=6.5mm,inner sep=0pt},
  edge/.style={line width=.8pt}]

\begin{scope}
  \path[draw=blue!70!black,fill=blue!8,rounded corners=8pt]
    (-3.55,-1.42) rectangle (-.80,1.42);
  \path[draw=black!60,fill=black!5] (0,0) circle (.54);
  \path[draw=red!70!black,fill=red!8,rounded corners=8pt]
    (.80,-1.42) rectangle (3.55,1.42);
  \node[vertex] (a1) at (-3,1) {$a_1$};
  \node[vertex] (a2) at (-3,-1) {$a_2$};
  \node[vertex] (a3) at (-1.35,0) {$a_3$};
  \node[vertex] (x1) at (0,0) {$x$};
  \node[vertex] (b3) at (1.35,0) {$b_3$};
  \node[vertex] (b1) at (3,1) {$b_1$};
  \node[vertex] (b2) at (3,-1) {$b_2$};
  \draw[edge] (a1)--(a2);
  \draw[edge] (a2)--(a3);
  \draw[edge] (a3)--(a1);
  \draw[edge] (b1)--(b2);
  \draw[edge] (b2)--(b3);
  \draw[edge] (b3)--(b1);
  \foreach \u in {a1,a2,a3,b1,b2,b3} \draw[edge] (x1)--(\u);
  \node[blue!70!black,font=\scriptsize] at (-1.17,1.18) {$K_3$};
  \node[font=\scriptsize] at (0,.78) {$K_1=\{x\}$};
  \node[red!70!black,font=\scriptsize] at (1.17,1.18) {$K_3$};
  \node[font=\bfseries] at (0,2.02) {(a) Clique tiles $K_3,K_1,K_3$};
  \node at (0,-1.87) {$w_L=w_x=w_R=1$,\quad cost $=3$};
\end{scope}

\begin{scope}[xshift=8.15cm]
  \path[draw=blue!70!black,fill=blue!8,rounded corners=8pt]
    (-3.55,-1.42) rectangle (.55,1.42);
  \path[draw=red!70!black,fill=red!8,rounded corners=8pt]
    (-.55,-1.42) rectangle (3.55,1.42);
  \node[vertex] (c1) at (-3,1) {$a_1$};
  \node[vertex] (c2) at (-3,-1) {$a_2$};
  \node[vertex] (c3) at (-1.35,0) {$a_3$};
  \node[vertex] (x2) at (0,0) {$x$};
  \node[vertex] (d3) at (1.35,0) {$b_3$};
  \node[vertex] (d1) at (3,1) {$b_1$};
  \node[vertex] (d2) at (3,-1) {$b_2$};
  \draw[edge,blue!70!black] (c1)--(c2);
  \draw[edge,blue!70!black] (c2)--(c3);
  \draw[edge,blue!70!black] (c3)--(c1);
  \foreach \u in {c1,c2,c3} \draw[edge,blue!70!black] (x2)--(\u);
  \draw[edge,red!70!black] (d1)--(d2);
  \draw[edge,red!70!black] (d2)--(d3);
  \draw[edge,red!70!black] (d3)--(d1);
  \foreach \u in {d1,d2,d3} \draw[edge,red!70!black] (x2)--(\u);
  \node[blue!70!black,font=\scriptsize] at (-1.17,1.18) {$K_4$};
  \node[red!70!black,font=\scriptsize] at (1.17,1.18) {$K_4$};
  \node[font=\bfseries] at (0,2.02) {(b) Clique tiles $K_4,K_4$};
  \node at (0,-1.87) {$s=t=1$,\quad $F(x)=2$,\quad cost $=2$};
\end{scope}
\end{tikzpicture}
\caption{Two tile families on the same graph.  The disjoint family in panel
(a) is a valid exact tiling but cannot give a bound below three.  The
overlapping clique family in panel (b) gives the sharp bound two; overcoverage
at the central vertex is harmless.  All three edges of each $K_3$ and all six
edges of each colored $K_4$ are displayed.}
\label{fig:coverage-intuition}
\end{figure}
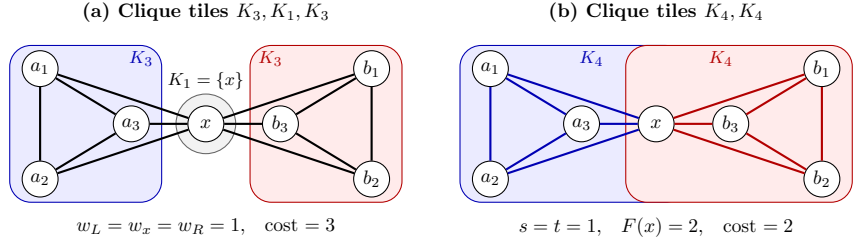

The dual of~\eqref{eq:abstract-tile-lp} is the fractional packing program
\begin{equation}
 \begin{aligned}
 \text{maximize}\quad&\sum_{x\in V}y_x,\\
 \text{subject to}\quad&\sum_{x\in V(T)}y_x\le\alpha(T)
       &&(T\in\mathcal T),\\
 &y_x\ge0&&(x\in V).
 \end{aligned}
 \label{eq:abstract-tile-dual}
\end{equation}
If all templates are cliques, this is the usual fractional clique-cover
bound; see~\cite{ScheinermanUllman1997}.  The applications here use translated
non-clique templates with known independence numbers.  This is distinct from
the Delsarte linear program, whose variables are distance distributions in an
association scheme~\cite{Delsarte1973}.

A fractional local cover is \emph{exact} if $F(x)=1$ for every vertex.  If all
positive-weight templates have common density
\[
 \dens(T)=\frac{\alpha(T)}{|V(T)|}=\rho,
\]
define the total overcoverage by
\[
 E(F)=\sum_{x\in V}(F(x)-1).
\]
Double counting gives the defect identity
\begin{equation}
 \cost(\mathcal T,w)=\rho\bigl(|V|+E(F)\bigr).
 \label{eq:abstract-defect}
\end{equation}
In particular, an exact cover has cost $\rho|V|$.  If $G$ has a proper
$q$-coloring, every induced subgraph has independence density at least $1/q$.
We call a template \emph{exact-density} when
\begin{equation}
 \alpha(T)=\frac{|V(T)|}{q}.
 \label{eq:exact-density-general}
\end{equation}
Together with Theorem~\ref{thm:weighted-tile-bound}, a cover by exact-density
templates therefore gives
\begin{equation}
 \alpha(G)\le\frac{|V|+E(F)}q.
 \label{eq:exact-density-bound}
\end{equation}
Thus $E(F)<q$ implies $\alpha(G)\le\lceil |V|/q\rceil$.  More generally, if
$B$ is an integer lower target, a cover of cost strictly smaller than $B+1$
proves $\alpha(G)\le B$.

We now isolate the finite-state step.  Fix a finite family
\[
 \mathcal H=\{U_j\subseteq\Delta_q(r_j):1\le j\le s\}.
\]
The translate with anchor $a\in\Delta_q(d-r_j)$ has vertex set $a+U_j$.
Only translation invariance and the independence number of each tile are used.
For a cap $C\ge0$, write
\[
 \tau_C(a)=(\min(a_1,C),\ldots,\min(a_q,C)).
\]
We assign the translate $a+U_j$ the weight $z_{j,\tau_C(a)}$.  The state is
initially ordered; coordinate symmetry will later permit sorting.
Fix also a starting degree $d_0$, a modulus $L$, a residue class
$\ell\pmod L$, a polynomial $P_\ell\in\mathbb Q[d]$ of degree at most $q-1$,
and a constant $\delta\in\mathbb Q$.  We seek nonnegative rational weights and one
constant $\delta'\le\delta$ such that, for every $d\ge d_0$ with
$d\equiv\ell\pmod L$,
\begin{equation}
 F_d(x)\ge1\quad(x\in\Delta_q(d)),
 \qquad
 \cost(F_d)=P_\ell(d)+\delta'.
 \label{eq:finite-state-goal}
\end{equation}

\begin{theorem}[Finite-state reduction]
\label{thm:finite-state-reduction}
For the data fixed above, the requirements in~\eqref{eq:finite-state-goal}
are equivalent to the feasibility of a finite rational linear system.
\end{theorem}

\begin{proof}
For $x\in\Delta_q(d)$, the coverage is
\begin{equation}
 F_d(x)=\sum_{j=1}^s\ 
 \sum_{\substack{u\in U_j\\u\le x}}
 z_{j,\tau_C(x-u)},
 \label{eq:state-coverage-formula}
\end{equation}
where $u\le x$ is coordinatewise.  Put $R=\max_j r_j$, $D=C+R$, and
\[
 \sigma_D(x)=(\min(x_1,D),\ldots,\min(x_q,D)).
\]
Suppose that $x_i\ge D$.  Since
$0\le u_i\le r_j\le R$, we have $x_i-u_i\ge C$, and therefore the $i$th
coordinate of $\tau_C(x-u)$ is $C$, independently of the actual value of
$x_i$.  If $x_i<D$, the state $\sigma_D(x)$ records $x_i$ exactly and hence
also records whether $u_i\le x_i$ and the value of
$\min(x_i-u_i,C)$.  Thus every summand and every eligibility condition in
\eqref{eq:state-coverage-formula} is determined by $j$, $u$, and
$\sigma_D(x)$.  There are only $(D+1)^q-D^q$ saturated ordered states, so all
saturated coverage inequalities form a finite system.  Every saturated state
$s$ is realized in every degree $d\ge |s|$: choose an index $i$ with $s_i=D$,
replace that coordinate by $D+d-|s|$, and leave all other coordinates equal
to those of $s$.  The resulting profile has degree $d$ and capped state $s$.
Thus the displayed state inequalities are necessary as well as sufficient,
also after restricting to one residue class.
If $x$ is unsaturated, then $x_i\le D-1$ for every $i$, so
$d=|x|\le q(D-1)$.  After fixing the residue class $\ell\pmod L$, we retain
only those unsaturated profiles with
$d_0\le |x|\le q(D-1)$ and $|x|\equiv\ell\pmod L$; profiles in the other
residue classes belong to their corresponding finite systems.  Hence only
finitely many unsaturated profiles remain in the system under consideration.

For the cost, fix a state $t$ with capped coordinate set $K$, $|K|=k$.  The
coordinates outside $K$ are fixed by $t$.  For $i\in K$, write
$a_i=C+y_i$ with $y_i\ge0$.  The anchor equation $|a|=d-r_j$ becomes
\[
 \sum_{i\in K}y_i=d-r_j-|t|.
\]
For $k\ge1$, stars and bars gives
\begin{equation}
 m_{j,t}(d)=
 \binom{d-r_j-|t|+k-1}{k-1},
 \label{eq:general-state-multiplicity}
\end{equation}
whenever the state is feasible.  For $k=0$, every anchor coordinate is fixed,
so the state occurs only in degree $d=r_j+|t|$, with multiplicity one there
and zero otherwise.  After enlarging the threshold $d_1$ chosen below if
necessary, all these finitely many contributions are handled among the
directly evaluated transition degrees and are omitted from the eventual
coefficient comparison.  For $k\ge1$, multiplying these multiplicities by the fixed local costs
$\alpha(G_q(r_j)[U_j])$ and the weights $z_{j,t}$ shows that the total cost is
a polynomial in $d$ of degree at most $q-1$, with coefficients that are
rational linear forms in the variables $z_{j,t}$.

All coverage coefficients are integers, and the coefficients of the
multiplicity polynomials are rational.  Choose a threshold $d_1\ge d_0$ after
which every capped anchor state having a capped coordinate is feasible.  For
a fixed residue class, impose the finite saturated inequalities, the finitely
many unsaturated inequalities, and nonnegativity.  For $d\ge d_1$, matching
the coefficients of $d^m$, $m\ge1$, between the cost and $P_\ell(d)$ gives
finitely many rational linear equations, and the constant coefficient is
written as the additional variable $\delta'$ with $\delta'\le\delta$.  For
each of the finitely many transition degrees
\[
 d_0\le d<d_1,\qquad d\equiv\ell\pmod L,
\]
include the directly evaluated coverage inequalities and the linear identity
$\cost(F_d)=P_\ell(d)+\delta'$.  These finitely many constraints are exactly the
finite rational system asserted in the statement.
\end{proof}

All families used below are invariant under permuting the $q$ coordinates.
Averaging the weights over these permutations preserves both coverage and
cost.  We may therefore assume that a weight depends only on the sorted capped
anchor profile, and it is enough to impose one coverage inequality for each
sorted capped vertex profile.

\begin{corollary}[Symmetry quotient]
\label{cor:symmetry-quotient}
Let $\Gamma\le S_q$ preserve the template library: for every
$\gamma\in\Gamma$ and every $j$, the set $\gamma U_j$ is another member
$U_{\gamma j}$ of the library, with the same local independence number.
Impose the equivariance condition
\[
 z_{\gamma j,\gamma t}=z_{j,t}.
\]
Then the variables and constraints in
Theorem~\ref{thm:finite-state-reduction} may be indexed by the corresponding
$\Gamma$-orbits of template--anchor pairs and vertex states, without changing
feasibility or cost.  In particular, for full coordinate symmetry the anchor
and vertex orbits may be represented by sorted capped tuples.
\end{corollary}

\begin{proof}
The group $\Gamma$ acts by automorphisms on every $G_q(d)$ and sends the
translate $a+U_j$ to $\gamma a+U_{\gamma j}$.  It preserves local
independence numbers, coverage incidences, and anchor multiplicities.
Consequently the coefficients of the ordered program are constant on the
stated orbits.  Identifying variables and constraints along these orbits gives
the quotient program.  Conversely, a quotient solution lifts by assigning
its orbit value to every ordered variable, and the lifted solution has the
same coverage and cost.
\end{proof}

\begin{remark}
Theorem~\ref{thm:finite-state-reduction} proves finite reducibility, not
feasibility.  Choosing a useful exact-density template library and finding a
feasible rational state vector remain the substantive steps.
\end{remark}

\subsection{Exact certificate verification}
\label{subsec:certificate-verification}

The finite certificates below are checked in three ways.  For a rational
local-cover certificate, the verifier regenerates the template incidences and
the complete capped state space from the definitions in the paper, parses the
stored rational vector, and checks nonnegativity, every coverage inequality,
and every coefficient of the asserted cost identity.  For a fixed-degree
orbit certificate it similarly reconstructs all partition-orbit rows and
checks the stored primal or dual vector over $\mathbb Q$.

For the finite nonexistence claims, a deterministic encoder constructs the
Boolean formula from the displayed combinatorial reduction.  The retained
reverse-unit-propagation (RUP) or deletion-resolution-asymmetric-tautology
(DRAT) derivation is then accepted only if a proof checker derives the empty
clause.  The DRAT format and the DRAT-trim checker used here are described
in~\cite{WetzlerHeuleHunt2014}.  The SAT solver's status is not a proof input.
Completeness comes
from the finite-state theorem for capped profiles and from the explicit
symmetry reductions stated in the exceptional finite cases.  The public
companion gives a theorem-to-artifact index, immutable file checksums, exact
verification commands, and their expected outputs
~\cite{ZabokritskiyCompanion2026}.

\begin{table}[htbp]
\centering
\caption{The central proof-grade computations.  File-level identifiers,
checksums, and commands are recorded in the cited companion.}
\label{tab:certificate-map}
\small
\begin{tabularx}{\linewidth}{@{}l >{\raggedright\arraybackslash}X
 >{\raggedright\arraybackslash}X@{}}
\toprule
result&finite reduction&independent check\\
\midrule
$\alpha_5(6)=42$&one CNF with $1594$ variables and $19581$ clauses
 &RUP derivation with $1043$ additions\\
$\alpha_7(5)=66$&fifty exhaustive type profiles; forty-six nontrivial CNFs
 &DRAT refutations; four direct contradictions\\
$\alpha_7(6)=133$&fifteen CNFs and the three remaining exact-integer cases
 &DRAT refutations and a deterministic exhaustive checker\\
$q=7$, $d\ge26$&$1092$ rational variables, $3003$ saturated rows, and
 $1547$ transition profiles&exact verification over $\mathbb Q$\\
\bottomrule
\end{tabularx}
\end{table}

\section{Three variables: planar local covers}
\label{sec:q3frac}

The ternary case is the first exact application of
Theorem~\ref{thm:finite-state-reduction}.  It gives a new proof of the
Geramita--Gregory--Roberts formula and displays, in two dimensions, the same
boundary-state mechanism used in the subsequent cases.
Specializing the additive coloring above
to \(A=\mathbb Z_3\) gives
\begin{equation}
 \alpha_3(d)\ge
 \left\lceil\frac{(d+1)(d+2)}6\right\rceil.
 \label{eq:q3-lower-frac}
\end{equation}
The degrees two and four are exceptional: their independence numbers are
three and six, respectively.

For degrees divisible by three the extremal set has a particularly attractive
geometric form.  Figure~\ref{fig:q3-machacek-pattern} shows the degree-twelve
example displayed by Machacek: the red vertices form the unique maximum
independent set.  Machacek's argument proves uniqueness by forcing the boundary
pattern and then moving inward~\cite{Machacek2021}.  We use the picture only as
geometric motivation.  The proof below is different: it bounds the size by a
weighted cover of translated smaller simplices and does not use uniqueness.

\begin{figure}
\centering
\begin{tikzpicture}[scale=.49]
\def\hh{0.8660254}
\foreach \aone in {0,...,12} {
  \pgfmathtruncatemacro{\remaining}{12-\aone}
  \foreach \atwo in {0,...,\remaining} {
    \pgfmathtruncatemacro{\athree}{12-\aone-\atwo}
    \pgfmathsetmacro{\xx}{0.5*(\athree-\atwo)}
    \pgfmathsetmacro{\yy}{\hh*\aone}
    \ifnum\aone>0
      \draw[black!58,line width=.3pt]
        (\xx,\yy)--({\xx-0.5},{\yy-\hh});
      \draw[black!58,line width=.3pt]
        (\xx,\yy)--({\xx+0.5},{\yy-\hh});
    \fi
    \ifnum\atwo>0
      \draw[black!58,line width=.3pt] (\xx,\yy)--({\xx+1},\yy);
    \fi
  }
}
\foreach \aone in {0,...,12} {
  \pgfmathtruncatemacro{\remaining}{12-\aone}
  \foreach \atwo in {0,...,\remaining} {
    \pgfmathtruncatemacro{\athree}{12-\aone-\atwo}
    \pgfmathtruncatemacro{\residue}{mod(\atwo+2*\athree,3)}
    \pgfmathsetmacro{\xx}{0.5*(\athree-\atwo)}
    \pgfmathsetmacro{\yy}{\hh*\aone}
    \ifnum\residue=0
      \filldraw[fill=red!82,draw=black,line width=.3pt]
        (\xx,\yy) circle[radius=.13];
    \fi
  }
}
\node[below left] at (-6,0) {$x_2^{12}$};
\node[below right] at (6,0) {$x_3^{12}$};
\node[above] at (0,{12*\hh}) {$x_1^{12}$};
\end{tikzpicture}
\caption{The unique maximum independent set of $G_3(12)$, shown in red.
It has $31$ vertices and is the zero class of
$a_2+2a_3\pmod 3$; compare~\cite[Figure~3]{Machacek2021}.}
\label{fig:q3-machacek-pattern}
\end{figure}
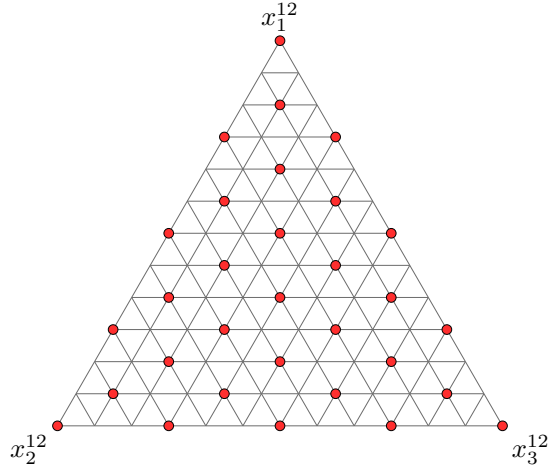

The residual degrees used in the eventual cover are
\begin{equation}
 \mathcal R_3=\{1,5,7,8\}.
 \label{eq:q3-residual-set}
\end{equation}
Exact finite computations give
\[
\begin{array}{c|rrrr}
 r&1&5&7&8\\ \hline
 |\Delta_3(r)|&3&21&36&45\\
 \alpha_3(r)&1&7&12&15
\end{array}
\]
and hence
\begin{equation}
 \alpha_3(r)=\frac{|\Delta_3(r)|}{3}
 \qquad(r\in\mathcal R_3).
 \label{eq:q3-exact-density}
\end{equation}
Equation~\eqref{eq:q3-exact-density} is the reason for choosing these tiles.
The additive coloring forces every tile to have independence density at least
$1/3$, and the weighted local-cover bound charges a tile of weight $w$ the
cost $w\alpha(H)=w|V(H)|/3$.  Hence an exact cover has total cost
$|\Delta_3(d)|/3$, while a cover whose total overcoverage is below three still
gives the required integer upper bound through~\eqref{eq:q3-defect}.  A tile
of density greater than $1/3$ would introduce a loss before the boundary is
considered.  Among the full simplices of residual degree at most eight, the
four displayed degrees are exactly those of density $1/3$; degrees two and
four are genuine exceptions, and the remaining degrees fail the necessary
divisibility condition.  The finite-state system then shows that translates
of these four shapes provide enough freedom near the boundary to keep the
defect below three.  The exact rational vector below certifies their
sufficiency; no uniqueness or minimality is claimed.

For \(a\in\Delta_3(d-r)\), the translate
\[
 a+\Delta_3(r)
\]
induces a copy of \(G_3(r)\) inside \(G_3(d)\).  Figure~\ref{fig:q3-translated-simplex}
shows a copy of \(G_3(5)\) inside \(G_3(12)\).  The anchor fixes a common
background exponent vector, while the residual degree-five mass moves inside
the smaller triangle.

\begin{figure}
\centering
\begin{tikzpicture}[scale=.56]
  \fill[blue!10]
    (4,1.73205) -- (9,1.73205) -- (6.5,6.06218) -- cycle;

  \foreach \i in {0,...,12}{
    \pgfmathtruncatemacro{\Jmax}{12-\i}
    \foreach \j in {0,...,\Jmax}{
      \pgfmathtruncatemacro{\k}{12-\i-\j}
      \pgfmathsetmacro{\xx}{\j+0.5*\k}
      \pgfmathsetmacro{\yy}{0.8660254*\k}
      \ifnum\i>0
        \pgfmathsetmacro{\xa}{\j+1+0.5*\k}
        \pgfmathsetmacro{\ya}{0.8660254*\k}
        \draw[gray!35,line width=.25pt] (\xx,\yy)--(\xa,\ya);
        \pgfmathsetmacro{\xb}{\j+0.5*(\k+1)}
        \pgfmathsetmacro{\yb}{0.8660254*(\k+1)}
        \draw[gray!35,line width=.25pt] (\xx,\yy)--(\xb,\yb);
      \fi
      \ifnum\j>0
        \pgfmathsetmacro{\xc}{\j-1+0.5*(\k+1)}
        \pgfmathsetmacro{\yc}{0.8660254*(\k+1)}
        \draw[gray!35,line width=.25pt] (\xx,\yy)--(\xc,\yc);
      \fi
      \fill[gray!58] (\xx,\yy) circle (0.7pt);

      \ifnum\i>1
        \ifnum\j>2
          \ifnum\k>1
             \fill[blue!75] (\xx,\yy) circle (1.15pt);
          \fi
        \fi
      \fi
    }
  }

  \draw[blue!70!black,line width=1.1pt]
    (4,1.73205) -- (9,1.73205) -- (6.5,6.06218) -- cycle;
  \node[below] at (6.5,-.45) {$G_3(12)$};
  \node[blue!70!black] at (6.5,3.15) {$a+G_3(5)$};
  \node[blue!70!black,align=center] at (9.9,4.8)
    {$a=(2,3,2)$\\$|a|=7$};
\end{tikzpicture}
\caption{A translated full simplex inside a larger triangular grid.  The
highlighted vertices are $a+\Delta_3(5)\subseteq\Delta_3(12)$, with anchor
$a=(2,3,2)$.}
\label{fig:q3-translated-simplex}
\end{figure}
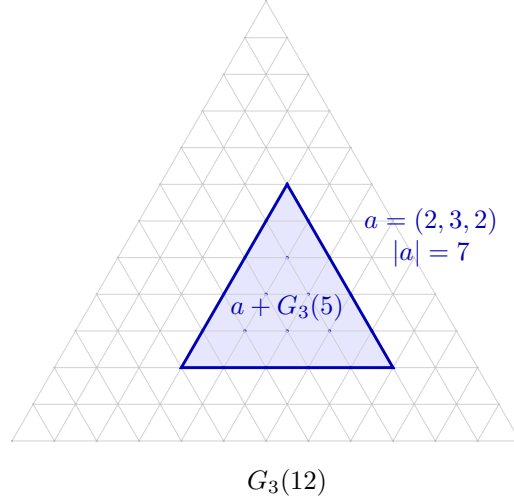

Assign weights \(w_{r,a}\ge0\), and let \(\tileind_{r,a}(x)\) denote the
indicator of the fixed-vertex statement \(x\in a+\Delta_3(r)\).  Define
\[
 F(x)=\sum_{r\in\mathcal R_3}
 \sum_{a\in\Delta_3(d-r)}w_{r,a}\tileind_{r,a}(x).
\]
If \(F(x)\ge1\) for every vertex, then every independent set \(I\) satisfies
\[
 |I|\le\sum_{r,a}w_{r,a}\alpha_3(r).
\]
Using \eqref{eq:q3-exact-density} and changing the order of summation gives
\begin{equation}
 \operatorname{cost}(F)-\frac{|\Delta_3(d)|}{3}
 =\frac13\sum_{x\in\Delta_3(d)}(F(x)-1).
 \label{eq:q3-defect}
\end{equation}
Thus total overcoverage below three proves the sharp checksum upper bound.

Set the anchor cap to five and let
\[
 \tau(a)=\operatorname{sort}(\min(a_1,5),\min(a_2,5),\min(a_3,5)).
\]
For \(d\ge21\), every anchor associated with a residual degree at most eight
has at least one capped coordinate.  There are twenty-one anchor states and
therefore
\[
 4\cdot21=84
\]
variables \(z_{r,\tau}\).  Capping vertex coordinates at
\(5+8=13\) produces 105 saturated vertex states.  The unsaturated profiles
have all coordinates at most twelve; among them, only 167 have total degree
at least twenty-one.  The cost is a quadratic polynomial in \(d\), so the
coefficients of \(d^2\) and \(d\) are matched exactly and the constant excess
is minimized.

\begin{proposition}[Ternary eventual certificate]
\label{prop:q3-eventual}
There is a nonnegative rational state vector with forty-two positive
coordinates such that all 105 saturated-state inequalities and all 167
relevant unsaturated inequalities hold, the two nonconstant cost coefficients
match exactly, and
\[
 \operatorname{cost}(F)-\frac{|\Delta_3(d)|}{3}
 =\delta_3=\frac57<1
 \qquad(d\ge21).
\]
Consequently the total coverage defect is \(3\delta_3<3\).
\end{proposition}

\begin{proof}
The archived verifier reconstructs all $84$ candidate variables from the four
residual degrees and the $21$ sorted anchor states.  It then checks the stored
rational vector against all $105$ saturated and $167$ unsaturated coverage
rows and checks the two nonconstant cost coefficients and the constant
$5/7$, all over $\mathbb Q$.  This is precisely the finite system obtained
above, so the exact checks prove the proposition; see the computational
companion~\cite{ZabokritskiyCompanion2026}.
\end{proof}

For the eighteen nonexceptional degrees below twenty-one, the same verifier
reconstructs the finite rational covers and checks every incidence and cost.
Degrees two and four are handled by the direct constructions below.  Thus no
maximum-independent-set solver status is used in the proof.

For $d=2$, the three corners $\{2e_i:i\in[3]\}$ form an independent set.
For the matching upper bound, give unit weight to each of the three cliques
$e_i+\Delta_3(1)$.  They cover every vertex and have total cost three.

For $d=4$, the six vertices
\[
 \{4e_i:i\in[3]\}\cup
 \{2e_i+2e_j:1\le i<j\le3\}
\]
form an independent set.  For the upper bound, give weight one to the three
cliques $3e_i+\Delta_3(1)$ and weight $1/2$ to each of the six cliques
$2e_i+e_j+\Delta_3(1)$, $i\ne j$.  Profiles of types $(4)$, $(3,1)$,
$(2,2)$, and $(2,1,1)$ receive coverage at least $1$, $3/2$, $1$, and $1$,
respectively, and the total cost is
\[
 3+6\cdot\frac12=6.
\]

\begin{theorem}[Exact ternary spreading number]
\label{thm:q3frac}
For every \(d\ge1\),
\[
 \alpha_3(d)=
 \begin{cases}
 3,&d=2,\\
 6,&d=4,\\
 \displaystyle\left\lceil\frac{(d+1)(d+2)}6\right\rceil,
 &d\notin\{2,4\}.
 \end{cases}
\]
\end{theorem}

\begin{proof}
The additive coloring gives the lower bound outside degrees two and four.
Proposition~\ref{prop:q3-eventual} gives the matching upper bound for
$d\ge21$, and the finite covers described above handle the remaining degrees.
The two exceptional values follow from their direct constructions.
\end{proof}

\section{Four variables: parity and boundary corrections}
\label{sec:q4frac}

We next apply the local-tile viewpoint to four variables.  This case has a
feature that does not occur for \(q=3\) or \(q=5\): the exact answer depends
on the parity of the degree.  Odd degrees admit an integral tiling by
\(K_4\)'s, while even degrees are handled by a finite-state fractional cover
using only the two full simplices \(G_4(1)\) and \(G_4(3)\).  These are not
two unrelated techniques: the odd construction is precisely the exact
\(F\equiv1\) specialization of the same LP framework, whereas the even case
requires nontrivial state-dependent weights.
For \(k\ge0\), write
\[
 M_4(2k+1)=\frac{(k+1)(k+2)(2k+3)}6
\]
and
\[
 M_4(2k)=\binom{k+3}{3}+\binom{k+1}{3}
        =\frac{(k+1)^3+2(k+1)}3.
\]
Equivalently, for even \(d\),
\begin{equation}
 M_4(d)=\frac{d^3}{24}+\frac{d^2}{4}+\frac{5d}{6}+1.
 \label{eq:q4-even-polynomial}
\end{equation}

Label the four coordinates by the elements of
\(H=\mathbb F_2^2\), and define
\[
 \chi_4(a)=\sum_{h\in H}a_h\,h,
 \qquad a=(a_h)_{h\in H}\in\Delta_4(d).
\]
A unit transfer changes the color by a nonzero element of \(H\), so every
fiber of \(\chi_4\) is independent.

If \(d=2k\), the zero fiber consists precisely of the vectors whose four
coordinates are all even or all odd.  Hence
\begin{equation}
 |\chi_4^{-1}(0)|
 =\binom{k+3}{3}+\binom{k+1}{3}
 =M_4(2k).
 \label{eq:q4-even-lower}
\end{equation}
For odd degree the four fibers have equal size, and therefore each has size
\[
 \frac14|\Delta_4(d)|=M_4(d).
\]

\subsection{Odd degree: an integral local tiling}
Let \(d\) be odd and let \(C_0=\chi_4^{-1}(0)\).  The parity pattern of a
vertex in \(C_0\) is either
\[
 (1,0,0,0)
 \qquad\text{or}\qquad
 (0,1,1,1),
\]
where the first coordinate is indexed by \(0\in H\).

For a center \(c\in C_0\) of the first type, define
\[
 Q(c)=\{c-e_0+e_h:h\in H\}.
\]
For a center of the second type, define
\[
 Q(c)=\{c+e_0-e_h:h\in H\}.
\]
The two orientations are dictated by parity rather than selected by a linear program.
In the first parity orbit the distinguished \(0\)-coordinate is the unique
odd coordinate, so one moves a unit out of that coordinate; in the second it
is the unique even coordinate, so the move is reversed.  Every vector of odd
coordinate sum has either one or three odd coordinates.  Its parity pattern
therefore determines the orientation and the transferred coordinate, after
which the vector itself determines the unique zero-color center.  Thus the two
parity orbits force exactly the two orientations needed for the partition.
Each \(Q(c)\) is a \(K_4\), and the resulting vertex sets form the disjoint
union
\begin{equation}
 \Delta_4(d)=\bigsqcup_{c\in C_0}Q(c).
 \label{eq:q4-odd-tiling}
\end{equation}
Every independent set meets each tile in at most one vertex, whereas \(C_0\)
meets every tile exactly once.  Therefore
\begin{equation}
 \alpha_4(d)=\frac14\binom{d+3}{3}=M_4(d)
 \qquad(d\text{ odd}).
 \label{eq:q4-odd-value}
\end{equation}
This is an integral local-tile certificate.  It uses the two orientations of
a degree-one simplex: the first family consists of ordinary translates of
\(G_4(1)\), while the second is the reflected orientation.  In the language
of the preceding terminology, all selected weights are one
and \(F\equiv1\).  Thus the associated linear program is solved by inspection: it is the
zero-defect case of the general framework.

\subsection{Even degree: exact-density templates}
For even degree we use ordinary translated full simplices only.  Set
\begin{equation}
 \mathcal R_4=\{1,3\}.
 \label{eq:q4-residual-set}
\end{equation}
The two local capacities are
\[
\begin{array}{c|cc}
 r&1&3\\ \hline
 |\Delta_4(r)|&4&20\\
 \alpha_4(r)&1&5
\end{array}
\]
and hence
\begin{equation}
 \alpha_4(r)=\frac{|\Delta_4(r)|}{4}
 \qquad(r\in\mathcal R_4).
 \label{eq:q4-exact-density}
\end{equation}

Here again the density condition is the first filter.  The degree-one tile is
the smallest local \(K_4\), while degree three is the smallest thicker full
simplex with the same optimal density.  Together, these two scales give the
state weights enough freedom near faces and lower-dimensional strata to
reproduce the unavoidable linear boundary excess in
\eqref{eq:q4-boundary-excess}.  The rational certificate below proves that
they suffice.  Their sufficiency was found by the finite-state search; no
uniqueness or minimality of \(\{1,3\}\) is asserted.

Assign a nonnegative weight \(w_{r,a}\) to every translate
\(a+\Delta_4(r)\), and let \(\tileind_{r,a}(x)\) be the indicator of
\(x\in a+\Delta_4(r)\).  Then
\[
 F(x)=
 \sum_{r\in\mathcal R_4}
 \sum_{a\in\Delta_4(d-r)}w_{r,a}\tileind_{r,a}(x).
\]
If \(F(x)\ge1\) for every vertex, then every independent set satisfies
\[
 |I|\le\operatorname{cost}(F)
 :=\sum_{r,a}w_{r,a}\alpha_4(r).
\]
As before, exact local density gives the defect identity
\begin{equation}
 \operatorname{cost}(F)-\frac14|\Delta_4(d)|
 =\frac14\sum_{x\in\Delta_4(d)}(F(x)-1).
 \label{eq:q4-defect}
\end{equation}
Unlike the odd case, the optimal even value is larger than the average by a
linear boundary term:
\begin{equation}
 M_4(d)-\frac14|\Delta_4(d)|
 =\frac{3d}{8}+\frac34
 \qquad(d\text{ even}).
 \label{eq:q4-boundary-excess}
\end{equation}
Thus the purpose of the linear program is not to make the defect bounded, but
to reproduce this exact boundary slack.

Set the anchor cap to \(C=3\) and define
\[
 \tau(a)=\operatorname{sort}
 \bigl(\min(a_1,3),\ldots,\min(a_4,3)\bigr).
\]
There are
\[
 \binom{6}{3}=20
\]
anchor states, and hence \(2\cdot20=40\) variables
\(z_{r,\tau}\).  Capping vertex coordinates at
\(C+\max\mathcal R_4=6\) gives \(84\) saturated states.  The certificate
also checks the \(18\) unsaturated sorted profiles with entries at most five
and total at least fifteen.
The anchor-state multiplicities are polynomials of degree at most three in
\(d\).  We match the coefficients of \(d^3,d^2,d\) to those in
\eqref{eq:q4-even-polynomial} and minimize the constant coefficient.

\begin{proposition}[Exact quaternary even-degree certificate]
\label{prop:q4-even-eventual}
There is a nonnegative rational state vector with twenty-nine positive coordinates
such that all \(84+18\) state inequalities hold and
\begin{equation}
 \operatorname{cost}(F)
 =\frac{d^3}{24}+\frac{d^2}{4}+\frac{5d}{6}+1
 \label{eq:q4-even-exact-cost}
\end{equation}
for every \(d\ge15\).  For every even \(d\ge16\), this polynomial is
exactly \(M_4(d)\), and hence gives the sharp upper bound.
\end{proposition}

\begin{proof}
The archived verifier reconstructs all forty state variables and the
twenty-nine positive rational entries.  It checks the $84$ saturated and $18$
unsaturated rows, the direct tile reconstruction, and every coefficient of
\eqref{eq:q4-even-exact-cost} over $\mathbb Q$.  Hence the displayed vector is
a cover with the asserted exact cost for every degree in the stated range;
see~\cite{ZabokritskiyCompanion2026}.
\end{proof}

For the remaining even degrees
\[
 d=2,4,6,8,10,12,14,
\]
the companion contains seven rational local-cover certificates using the
same two residual degrees, with costs
\[
 4,11,24,45,76,119,176,
\]
respectively.  The verifier reconstructs all $135$ finite state-orbit rows
and checks the incidences and costs exactly.  These costs are precisely
$M_4(d)$.

\begin{theorem}[Exact quaternary spreading number]
\label{thm:q4frac}
For every \(k\ge0\),
\[
 \alpha_4(2k+1)
 =\frac{(k+1)(k+2)(2k+3)}6,
\]
and
\[
 \alpha_4(2k)
 =\binom{k+3}{3}+\binom{k+1}{3}.
\]
\end{theorem}

\begin{proof}
The additive coloring gives the lower bounds.  For odd degree, the integral
\(K_4\)-tiling \eqref{eq:q4-odd-tiling} gives the matching upper bound.  For
even \(d\ge16\), Proposition~\ref{prop:q4-even-eventual} gives
\(|I|\le M_4(d)\).  The seven exact finite covers give the same bound for the
remaining positive even degrees.  For \(d=0\), the graph has one vertex, in
agreement with the displayed formula.  Hence the additive constructions are
optimal in every degree.
\end{proof}

\begin{remark}
The parity split is intrinsic.  In odd degree the optimum equals the average
color-class size and is witnessed by an exact integral tiling using both
orientations of \(G_4(1)\).  In even degree the optimum exceeds the average by
\(3d/8+3/4\); the fractional \(\{1,3\}\)-tile cover reproduces precisely this
boundary contribution.
\end{remark}

\section{The case \texorpdfstring{$q=5$}{q=5}: exact local tiles}
\label{sec:q5}

The cyclic coloring gives
\begin{equation}
 \alpha_5(d)\ge
 \left\lceil\frac{1}{5}\binom{d+4}{4}\right\rceil.
 \label{eq:q5-lower}
\end{equation}
There are two exceptions to equality.  The sets
\[
 \{2e_i:i\in[5]\}
\]
and
\begin{equation}
 \{4e_i:i\in[5]\}
 \ \cup\ \{2e_i+2e_j:1\le i<j\le5\}
 \ \cup\ \{(0,1,1,1,1)\}
 \label{eq:q5-d4-exception}
\end{equation}
are independent and have sizes $5$ and $16$, respectively.  We prove that
these are the only exceptions.

Both parts of the proof use coordinate-permutation symmetry, so we fix the
notation before its first use.  For \(v=(v_1,\ldots,v_q)\in
\mathbb Z_{\ge0}^q\), let
\[
 \mathcal O_q(v)
 =\{(v_{\sigma(1)},\ldots,v_{\sigma(q)}):\sigma\in S_q\}
\]
be its orbit under coordinate permutations, where \(S_q\) is the symmetric
group on \([q]\).  If \(\lambda=(\lambda_1,\ldots,\lambda_k)\) is a partition
with \(k\le q\), we identify it with
\((\lambda_1,\ldots,\lambda_k,0,\ldots,0)\) and write
\(\mathcal O_q(\lambda)\); for \(k>q\), set
\(\mathcal O_q(\lambda)=\varnothing\).  Equivalently,
\(u\in\mathcal O_q(\lambda)\) exactly when its positive coordinates, arranged
in nonincreasing order, form \(\lambda\).  For example,
\[
 \mathcal O_q(2,1)=\{2e_i+e_j:i,j\in[q],\ i\ne j\},\qquad
 \mathcal O_q(2,2)=\{2e_i+2e_j:i<j\},
\]
so these sets have \(q(q-1)\) and \(\binom q2\) elements, respectively.
We abbreviate a partition with \(k\) parts equal to one by \((1^k)\).  The
same notation applies to a sorted capped anchor state
\(\tau\): \(\mathcal O_5(\tau)\) is its set of distinct coordinate
permutations.

Figure~\ref{fig:orbit-summary} gives the complete geometric example used
throughout the rest of the paper.  It first decomposes residual degree three
into coordinate orbits and then translates the same local pattern by an
anchor.  Later orbit-union tiles merely select some of these colored layers
before applying the translation.

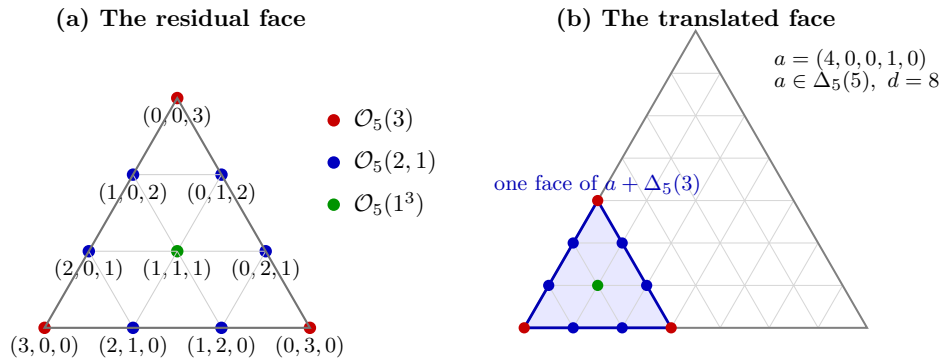
\begin{figure}[b]
\centering
\begin{tikzpicture}[scale=.90,transform shape,every node/.style={font=\normalsize}]
\begin{scope}
  \node[font=\bfseries] at (2.0,4.52) {(a) The residual face};
  \foreach \i in {0,...,3}{
    \pgfmathtruncatemacro{\Jmax}{3-\i}
    \foreach \j in {0,...,\Jmax}{
      \pgfmathtruncatemacro{\k}{3-\i-\j}
      \pgfmathsetmacro{\xx}{1.30*(\j+0.5*\k)}
      \pgfmathsetmacro{\yy}{1.30*0.8660254*\k}
      \ifnum\i>0
        \draw[gray!38,line width=.35pt] (\xx,\yy)--++(1.30,0);
        \draw[gray!38,line width=.35pt] (\xx,\yy)--++(.65,1.1258);
      \fi
      \ifnum\j>0
        \draw[gray!38,line width=.35pt] (\xx,\yy)--++(-.65,1.1258);
      \fi
      \def\pointcolor{blue!75!black}
      \ifnum\i=3\def\pointcolor{red!78!black}\fi
      \ifnum\j=3\def\pointcolor{red!78!black}\fi
      \ifnum\k=3\def\pointcolor{red!78!black}\fi
      \ifnum\i=1\ifnum\j=1\ifnum\k=1
        \def\pointcolor{green!58!black}
      \fi\fi\fi
      \fill[\pointcolor] (\xx,\yy) circle (2.6pt);
      \node[font=\scriptsize,fill=white,fill opacity=.82,text opacity=1,
        inner sep=.45pt,anchor=north,yshift=-2.6pt]
        at (\xx,\yy) {$({\i},{\j},{\k})$};
    }
  }
  \draw[black!55,line width=.8pt] (0,0)--(3.9,0)--(1.95,3.377)--cycle;
  \fill[red!78!black] (4.25,3.05) circle (2.6pt);
  \node[anchor=west] at (4.42,3.05) {$\mathcal O_5(3)$};
  \fill[blue!75!black] (4.25,2.43) circle (2.6pt);
  \node[anchor=west] at (4.42,2.43) {$\mathcal O_5(2,1)$};
  \fill[green!58!black] (4.25,1.81) circle (2.6pt);
  \node[anchor=west] at (4.42,1.81) {$\mathcal O_5(1^3)$};
\end{scope}

\begin{scope}[xshift=7.05cm]
  \node[font=\bfseries] at (2.52,4.52) {(b) The translated face};
  \fill[blue!9]
    (0,0)--(2.16,0)--(1.08,1.8706)--cycle;
  \foreach \ii in {0,...,7}{
    \pgfmathtruncatemacro{\JJmax}{7-\ii}
    \foreach \jj in {0,...,\JJmax}{
      \pgfmathtruncatemacro{\kk}{7-\ii-\jj}
      \pgfmathsetmacro{\xx}{.72*(\jj+.5*\kk)}
      \pgfmathsetmacro{\yy}{.72*.8660254*\kk}
      \ifnum\ii>0
        \draw[gray!32,line width=.27pt] (\xx,\yy)--++(.72,0);
        \draw[gray!32,line width=.27pt] (\xx,\yy)--++(.36,.6235);
      \fi
      \ifnum\jj>0
        \draw[gray!32,line width=.27pt] (\xx,\yy)--++(-.36,.6235);
      \fi
    }
  }
  \draw[black!50,line width=.75pt] (0,0)--(5.04,0)--(2.52,4.365)--cycle;
  \draw[blue!70!black,line width=1.05pt]
    (0,0)--(2.16,0)--(1.08,1.8706)--cycle;
  \foreach \i in {0,...,3}{
    \pgfmathtruncatemacro{\Jmax}{3-\i}
    \foreach \j in {0,...,\Jmax}{
      \pgfmathtruncatemacro{\k}{3-\i-\j}
      \pgfmathsetmacro{\xx}{.72*(\j+.5*\k)}
      \pgfmathsetmacro{\yy}{.72*.8660254*\k}
      \def\pointcolor{blue!75!black}
      \ifnum\i=3\def\pointcolor{red!78!black}\fi
      \ifnum\j=3\def\pointcolor{red!78!black}\fi
      \ifnum\k=3\def\pointcolor{red!78!black}\fi
      \ifnum\i=1\ifnum\j=1\ifnum\k=1
        \def\pointcolor{green!58!black}
      \fi\fi\fi
      \fill[\pointcolor] (\xx,\yy) circle (2.35pt);
    }
  }
  \node[blue!70!black,font=\scriptsize,align=center] at (1.08,2.10)
    {one face of $a+\Delta_5(3)$};
  \node[anchor=west,font=\scriptsize,align=left] at (3.55,3.78)
    {$a=(4,0,0,1,0)$\\$a\in\Delta_5(5),\ d=8$};
\end{scope}
\end{tikzpicture}
\caption{Coordinate orbits and a translated placement.  In panel (a), each
label $(i,j,k)$ denotes $(i,j,k,0,0)\in\Delta_5(3)$; the three colors are
the orbit layers $\mathcal O_5(3)$, $\mathcal O_5(2,1)$, and
$\mathcal O_5(1^3)$.  Panel (b) adds the displayed anchor to the same orbit
layers and shows their translated face inside a coordinate slice of
$\Delta_5(8)$.  Only one two-dimensional face of the small four-dimensional
simplex $a+\Delta_5(3)$ is drawn; permuting all five coordinates generates
the complete orbits.}
\label{fig:orbit-summary}
\end{figure}
\subsection{A symmetric full-simplex cover for \texorpdfstring{$d\ge35$}{d >= 35}}

The eventual argument uses only translated smaller copies of the same graph.
Set
\begin{equation}
 \mathcal R_5=\{1,3,6,7,8,9\}.
 \label{eq:q5-residual-set}
\end{equation}
For these residual degrees the exact finite values are
\[
\begin{array}{c|rrrrrr}
 r&1&3&6&7&8&9\\ \hline
 |\Delta_5(r)|&5&35&210&330&495&715\\
 \alpha_5(r)&1&7&42&66&99&143.
\end{array}
\]
The finite arguments later in this section establish all six values.  Thus
\begin{equation}
 \alpha_5(r)=\frac{|\Delta_5(r)|}{5}
 \qquad(r\in\mathcal R_5).
 \label{eq:q5-exact-density}
\end{equation}
These are precisely the exact-density full simplices of residual degree at
most nine.  Every tile is therefore a genuine smaller discrete simplex, it
retains full coordinate symmetry, and it pays no local density penalty.

For $a\in\Delta_5(d-r)$, the translate $a+\Delta_5(r)$ induces a copy of
$G_5(r)$ inside $G_5(d)$.  Hence every independent set $I$ satisfies
\[
 |I\cap(a+\Delta_5(r))|\le\alpha_5(r).
\]
By~\eqref{eq:q5-exact-density} and the defect identity
\eqref{eq:abstract-defect}, any fractional cover by these tiles with total
overcoverage below five proves the upper bound in~\eqref{eq:q5-lower}.

Give the translate $a+\Delta_5(r)$ a weight depending only on $r$ and on
\[
 \tau(a)=\operatorname{sort}_{\uparrow}
 \bigl(\min(a_1,6),\ldots,\min(a_5,6)\bigr),
\]
where \(\operatorname{sort}_{\uparrow}\) arranges the entries in
nondecreasing order.  There are $210$ sorted anchor states and six residual
degrees, hence $1260$
variables.  Since the largest residual degree is nine, coverage is determined
by the $3876$ saturated sorted vertex profiles obtained by capping coordinates
at fifteen and by the $6021$ unsaturated sorted profiles whose coordinates are
at most fourteen and whose total is at least thirty-five.
If an anchor state $\tau$ has $k$ capped entries and residual degree $r$, then
the multiplicity of its coordinate orbit in degree $d$ is
\begin{equation}
 m_{r,\tau}(d)=|\mathcal O_5(\tau)|
 \binom{d-r-|\tau|+k-1}{k-1},
 \label{eq:q5-state-multiplicity}
\end{equation}
with the usual value zero when infeasible.  Consequently the cost of a fixed
state vector is a polynomial in $d$ of degree at most four.

\begin{proposition}[Exact full-simplex certificate]
\label{prop:q5-eventual}
There is a nonnegative rational state vector with $933$ positive coordinates
among the $1260$ variables which covers every vertex for every $d\ge35$ and
satisfies
\[
 \cost(F_d)-\frac15\binom{d+4}{4}=\delta_5,
 \qquad 0<\delta_5<1.
\]
Consequently
\[
 \alpha_5(d)=\left\lceil\frac15\binom{d+4}{4}\right\rceil
 \qquad(d\ge35).
\]
\end{proposition}

\begin{proof}
The self-describing rational vector is archived in the computational
companion~\cite{ZabokritskiyCompanion2026}.  A verifier reconstructed
independently from the definitions above checks the variable ordering, all
$3876+6021$ coverage rows, and all four nonconstant coefficients of the cost
polynomial using a common integer denominator.  It obtains an exact rational constant with $0<\delta_5<1$
($\delta_5\approx0.9574841054$); the decimal is displayed only for
orientation and is not used in the proof.
The verifier also evaluates the actual multiplicities in the transition
degrees before all capped states are automatically feasible.  These exact
checks establish the asserted cover for every $d\ge35$.
\end{proof}

\subsection{Finite degrees and the isolated degree six}

Using the orbit notation fixed above, we use the following six induced
subgraphs:
\begin{equation}
\begin{array}{c|c|c|c}
H&V(H)&|V(H)|&\alpha(H)\\ \hline
J_1&\mathcal O_5(1)&5&1\\
J_2&\mathcal O_5(1^2)&10&2\\
J_3&\mathcal O_5(1^3)&10&2\\
P_3&\mathcal O_5(2,1)\cup \mathcal O_5(1^3)&30&6\\
J_4&\mathcal O_5(1^4)&5&1\\
R_{4c}&\mathcal O_5(3,1)\cup \mathcal O_5(2,1,1)\cup \mathcal O_5(1^4)&55&11.
\end{array}
\label{eq:q5-six-tiles}
\end{equation}

\begin{lemma}
\label{lem:q5-six-tiles}
The independence numbers in~\eqref{eq:q5-six-tiles} are exact.  In
particular, every one of the six templates has independence density $1/5$.
\end{lemma}

\begin{proof}
$J_1$ and $J_4$ are $K_5$.  An independent set in $J_2$ is a matching in
$K_5$, and complementation identifies $J_3$ with the same Johnson graph.
For $P_3$, at most one word $2e_i+e_j$ can be chosen for each fixed doubled
coordinate $i$.  At most two triples can be chosen; if two are chosen, they
meet in one coordinate and permit at most four compatible words of type
$(2,1)$.  Thus $\alpha(P_3)\le6$.

For $R_{4c}$, fix the coordinate containing the large entry.  The selected
words of types $(3,1)$ and $(2,1,1)$ contribute at most two in each of the
five coordinate layers: the latter induce $J(4,2)$, and selecting a word of
type $(3,1)$ lowers their capacity from two to one.  The $J_4$ layer adds at
most one, so $\alpha(R_{4c})\le11$.  Conversely, on every template the cyclic
checksum has five equinumerous independent fibers, since its degree is
nonzero modulo five.  This gives all matching lower bounds.
\end{proof}

Put
\[
 B_5(d)=
 \begin{cases}
  5,&d=2,\\
  16,&d=4,\\
  \left\lceil\binom{d+4}{4}/5\right\rceil,&\text{otherwise}.
 \end{cases}
\]

\begin{proposition}[Finite rational certificates]
\label{prop:q5-finite}
For every $2\le d\le29$ with $d\ne6$, there is a rational fractional cover
by translates of the six templates in~\eqref{eq:q5-six-tiles} whose cost is
strictly smaller than $B_5(d)+1$.  Consequently
$\alpha_5(d)\le B_5(d)$ in this range.
\end{proposition}

\begin{proof}
The companion stores one rational vector for each of the $27$ degrees.  The
verifier reconstructs every translated-tile incidence and checks all
coverage inequalities and costs as integer numerators over the stored common
denominator.  The weighted local-cover bound and integrality give the
conclusion~\cite{ZabokritskiyCompanion2026}.
\end{proof}

The five transition degrees $30\le d\le34$ are closed by a second exact
certificate using only $J_1,J_2,P_3,R_{4c}$.  It has $840$ state variables,
$401$ positive entries, $715$ saturated and $110$ unsaturated profile rows,
and exact defect $4/5$.  Its verifier reconstructs every row and the cost
identity over $\mathbb Q$.  Although this compact vector is valid for all
$d\ge30$, here it is used only for the five transition degrees so that the
structurally simpler full-simplex construction remains the eventual proof
\cite{ZabokritskiyCompanion2026}.

It remains to settle degree six.  Replacing a possible selected neighbor of a
corner $6e_i$ by the corner shows that a maximum independent set may be
assumed to contain all five corners and no word of type $(5,1)$.  A
hypothetical independent set of size $43$ would therefore leave an independent
set $J$ of size $38$ on the remaining $185$ vertices.

The five translated-tile families below cover these $185$ vertices exactly.  For each displayed anchor type $\lambda$, every anchor $a\in\mathcal O_5(\lambda)$ is used:
\begin{equation}
\begin{array}{c|c|c}
\text{tile}&\text{anchor type }\lambda&\text{weight}\\ \hline
J_1&(4,1)&1/3\\
J_1&(3,2)&1/2\\
J_1&(2,2,1)&1/6\\
J_2&(3,1)&1/6\\
J_2&(2,1,1)&1/6.
\end{array}
\label{eq:q5-d6-cover}
\end{equation}
Its cost is $115/3$.  After multiplication by six, its total capacity is
$230$, whereas the $38$ selected vertices consume $228$.  If $x_\lambda$
denotes the number of selected words of type $\lambda$, the five nonnegative
integral capacity slacks are
\begin{align*}
 s_{41}&=20-(x_{42}+2x_{411}),\\
 s_{32}&=20-(x_{42}+2x_{33}+x_{321}),\\
 s_{221}&=30-(x_{321}+3x_{222}+2x_{2211}),\\
 s_{31}&=40-(x_{42}+2x_{411}+x_{321}+3x_{3111}),\\
 s_{211}&=60-(x_{321}+3x_{222}+3x_{3111}
                    +4x_{2211}+6x_{21111}),
\end{align*}
and exact coverage gives
\begin{equation}
 2s_{41}+3s_{32}+s_{221}+s_{31}+s_{211}=2.
 \label{eq:q5-d6-slack}
\end{equation}
In particular,
\begin{equation}
 s_{32}=0,
 \qquad x_{42}+2x_{33}+x_{321}=20.
 \label{eq:q5-d6-shadow}
\end{equation}

\begin{lemma}
\label{lem:q5-d6-reduction}
Every $J$ of size $38$ satisfying~\eqref{eq:q5-d6-slack} contains a word of
type $(4,2)$.
\end{lemma}

\begin{proof}
First, $x_{411}+x_{3111}\le10$: in each fixed large-coordinate layer the
type-$(4,1,1)$ vertices induce $J(4,2)$ and contribute at most two, while a
selected type-$(3,1,1,1)$ vertex lowers their capacity to one.  Summing over
the five layers proves the claim.

Suppose $x_{42}=0$.  Since $s_{41}$ is even and
\eqref{eq:q5-d6-slack} gives $2s_{41}\le2$, we get $x_{411}=10$.
Equation~\eqref{eq:q5-d6-shadow} becomes $x_{321}+2x_{33}=20$.  If
$x_{3111}=0$, the size equation gives
$x_{222}+x_{2211}+x_{21111}=8+x_{33}$.  The last profile type, $\mathcal O_5(2,1^4)$, is a clique, so
$x_{21111}\le1$, which forces $s_{221}<0$.  If $x_{3111}=1$, then
$s_{31}=2x_{33}-3$; hence $0\le s_{31}\le2$ forces $x_{33}=2$ and
$x_{321}=16$.  The size equation now gives
$x_{222}+x_{2211}+x_{21111}=9$, again forcing $s_{221}<0$.  Thus
$x_{3111}\ge2$, contrary to $x_{411}+x_{3111}\le10$.
\end{proof}

By coordinate symmetry we may fix $(0,0,0,2,4)\in J$.  Encode the remaining
question as a Boolean formula with one variable for each of the $185$
vertices.  It has the adjacency clauses, a standard totalizer requiring at
least $38$ selected vertices, the twenty labeled shadow clauses forced by
\eqref{eq:q5-d6-shadow}, and the unit clause selecting $(0,0,0,2,4)$.  The
resulting CNF has $1594$ variables, including the totalizer auxiliaries, and
$19581$ clauses.

\begin{proposition}[Certified degree-six refutation]
\label{prop:q5-d6-rup}
The reduced CNF is unsatisfiable.  It has a reverse-unit-propagation proof of
$1043$ clause additions ending with the empty clause.
\end{proposition}

\begin{proof}
The companion contains the deterministic encoder, the generated formula, the
retained RUP derivation, and a standalone checker.  The checker validates each
of the $1043$ additions by unit propagation against the preceding formula and
accepts only after deriving the empty clause.  Thus the solver that discovered
the refutation is outside the trusted proof chain
\cite{ZabokritskiyCompanion2026}.
\end{proof}

It follows that $\alpha_5(6)\le42$; equality follows from the cyclic coloring,
since $|\Delta_5(6)|=210$.

\begin{theorem}[Exact five-variable spreading number]
\label{thm:q5}
For every $d\ge1$,
\[
 \alpha_5(d)=
 \begin{cases}
  5,&d=2,\\
  16,&d=4,\\
  \displaystyle\left\lceil\frac15\binom{d+4}{4}\right\rceil,
     &d\notin\{2,4\}.
 \end{cases}
\]
\end{theorem}

\begin{proof}
The graph $G_5(1)$ is $K_5$.  For larger degrees the additive coloring and
the two displayed exceptional sets give the lower bounds.
Proposition~\ref{prop:q5-finite} gives the upper bounds for
$2\le d\le29$, $d\ne6$; Proposition~\ref{prop:q5-d6-rup} and the preceding
reduction handle $d=6$; the transition certificate handles $30\le d\le34$;
and Proposition~\ref{prop:q5-eventual} handles every $d\ge35$.
\end{proof}

\section{The case \texorpdfstring{$q=7$}{q=7}: the complete solution}
\label{sec:q7}

The cyclic checksum fibers give the lower bound
\[
 L_7(d):=\left\lceil\frac17\binom{d+6}{6}\right\rceil.
\]
Every vector used in this section has seven coordinates.  Let
$e_1,\ldots,e_7$ be the standard unit vectors of $\mathbb Z^7$.  We use the
coordinate-orbit notation fixed at the start of Section~\ref{sec:q5}; thus
$\mathcal O_7(\lambda)$ is the set of seven-coordinate permutations of the
partition $\lambda$ padded by zeros.

In the eventual three- and five-variable proofs, translated full residual
simplices provided enough local information.  For seven variables they do not
distinguish the boundary layers finely enough to obtain the sharp bound in all
degrees.  We therefore decompose each small simplex $\Delta_7(r)$ into its
coordinate-permutation orbits $\mathcal O_7(\lambda)$ and build a tile by
taking a union of selected orbits.  This preserves the full symmetry among
the seven coordinates while allowing different boundary profiles to be
covered separately.  The translation mechanism is exactly the one shown for
$q=5$ in Figure~\ref{fig:orbit-summary}: for a fixed anchor $a$ and an orbit
union $U\subseteq\Delta_7(r)$, the placement satisfies
$a+U\subseteq a+\Delta_7(r)\subseteq\Delta_7(d)$.  The orbit unions defined
next provide the extra flexibility needed for the complete seven-variable
cover.

We now define the tiles before summarizing them.  Let
\[
\begin{aligned}
 U_1&=\mathcal O_7(1),&
 U_2&=\mathcal O_7(1^2),\\
 U_3&=\mathcal O_7(1^3),&
 U_4&=\mathcal O_7(2,1)\cup\mathcal O_7(1^3),\\
 U_5&=\mathcal O_7(1^4),&
 U_6&=\mathcal O_7(3,1)\cup\mathcal O_7(2,1,1)
       \cup\mathcal O_7(1^4),\\
 U_7&=\mathcal O_7(1^5),&
 U_8&=\mathcal O_7(1^6),\\
 U_9&=\mathcal O_7(3)\cup\mathcal O_7(2,1),&
 U_{10}&=\mathcal O_7(4)\cup\mathcal O_7(3,1),\\
 U_{11}&=\mathcal O_7(3,1)\cup\mathcal O_7(2,1,1),\\
 U_{12}&=\mathcal O_7(3,1)\cup\mathcal O_7(2,2)
          \cup\mathcal O_7(2,1,1),\\
 U_{13}&=\Delta_7(5).
\end{aligned}
\]
Their residual degrees are
\[
 (r_1,\ldots,r_{13})=(1,2,3,3,4,4,5,6,3,4,4,4,5),
\]
and the corresponding base tiles are
\[
 H_t=G_7(r_t)[U_t],
 \qquad
 \mathcal H_7=\{H_1,\ldots,H_{13}\}.
\]
This is the same construction used in the preceding sections.  For each
$t$ and each anchor $a\in\Delta_7(d-r_t)$, the placed tile has vertex set
$a+U_t$, and the full family of placements is
\begin{equation}
 \mathcal T_7(d)=
 \{a+H_t:1\le t\le13,\ a\in\Delta_7(d-r_t)\}.
 \label{eq:q7-placed-family}
\end{equation}
Both $a$ and every $u\in U_t$ have length seven; only their total masses are
$d-r_t$ and $r_t$.  Orbit notation records the pattern of the residual vector
$u$ and does not shorten the anchor or change the ambient dimension.  The
symbol $a$ is local to a placement: different tile degrees have their own
anchor sets, and every admissible anchor is used.  We later assign a weight
$w_{t,a}$ to each placement and group equal weights only after capping and
sorting the seven coordinates of $a$.

For a concrete example, $U_4$ is the union of the $42$ vectors of type
$(2,1)$ and the $\binom73=35$ vectors of type $(1^3)$, so $H_4$ has $77$
vertices.  The more elaborate degree-four set is
\[
\begin{aligned}
 U_{12}
 ={}&\{3e_i+e_j:i\ne j\}
 \cup\{2e_i+2e_j:i<j\}\\
 &\cup\{2e_i+e_j+e_k:i,j,k\text{ distinct},\ j<k\}.
\end{aligned}
\]
Its three parts have $42$, $21$, and $105$ vertices.  A placement of this
tile is $a+U_{12}$ with $a\in\Delta_7(d-4)$: the same seven-coordinate
background vector $a$ is added to every residual vector displayed above.

The first eight orbit-union tiles capture the natural low-degree boundary
layers, but the uniform certificate used below also needs five corrections.
The tiles $H_9,\ldots,H_{13}$ retain independence density $1/7$ and add the
corner, edge, and interior scales needed by the single all-residue-class
cover.  The following table collects the definitions and local capacities.

\begin{table}[htbp]
\centering
\caption{The family $\mathcal H_7$.  The line separates the first eight tiles
from the five boundary corrections.}
\label{tab:q7-tiles}
\begin{tabular}{@{}c c l r r@{}}
\toprule
tile&degree $r$&profile types in $U$&$|U|$&$\alpha(G_7(r)[U])$\\
\midrule
$H_1$&1&$(1)$&7&1\\
$H_2$&2&$(1^2)$&21&3\\
$H_3$&3&$(1^3)$&35&7\\
$H_4$&3&$(2,1)\cup(1^3)$&77&11\\
$H_5$&4&$(1^4)$&35&7\\
$H_6$&4&$(3,1)\cup(2,1,1)\cup(1^4)$&182&26\\
$H_7$&5&$(1^5)$&21&3\\
$H_8$&6&$(1^6)$&7&1\\
\midrule
$H_9$&3&$(3)\cup(2,1)$&49&7\\
$H_{10}$&4&$(4)\cup(3,1)$&49&7\\
$H_{11}$&4&$(3,1)\cup(2,1,1)$&147&21\\
$H_{12}$&4&$(3,1)\cup(2,2)\cup(2,1,1)$&168&24\\
$H_{13}$&5&all profiles in $\Delta_7(5)$&462&66\\
\bottomrule
\end{tabular}
\end{table}

The five corrections also have a direct geometric meaning.  The vertex set of
$H_9$ partitions, according to the heavy coordinate, into seven upward
$K_7$'s.  The induced graph may also contain edges between different parts,
which can only strengthen the resulting upper bound; $H_{10}$ is the analogous
correction in residual degree four.  The tiles $H_{11}$ and $H_{12}$ join the degree-four
layers concentrated on two or three coordinates, while $H_{13}=G_7(5)$ is a
full residual simplex and supplies a thicker interior scale.  Thus the larger
family adds exact-density shapes at precisely the corner, edge, and interior
boundary scales missed by $H_1,\ldots,H_8$.  The eventual construction is one
weighted cover formed from all placements in~\eqref{eq:q7-placed-family}, not
thirteen separate decompositions of $G_7(d)$.

This example also shows why the method does not immediately settle every
prime $q$.  As $q$ grows, more boundary profile types appear and the tiles
needed to cover them become less simple.  The finite system also grows
quickly: with $s$ tile types and anchor cap $C$, the saturated anchor states
alone give
\[
 s\binom{C+q-1}{q-1}
\]
weight variables.  Thus the reduction remains finite, but finding a useful
family and solving the resulting system become harder.  We nevertheless
expect Conjecture~\ref{conj:prime-checksum-optimality} to hold for every prime
$q$: the checksum formula should be correct for all sufficiently large $d$,
while finitely many small degrees may remain exceptional, as they do for
$q=3,5,7$.

\begin{lemma}[Local capacities]
\label{lem:q7-local-tiles}
The thirteen independence numbers in Table~\ref{tab:q7-tiles} are exact.
\end{lemma}

\begin{proof}
The squarefree layers $H_1,H_2,H_3,H_5,H_7,H_8$ give the clique, matching,
pair-packing, and complementary pair-packing bounds; the Fano lines attain the
values for $H_3$ and $H_5$.  For $H_4$, if an independent collection contains
$s$ triples, their $3s$
pairs are unavailable to the directed type-$(2,1)$ vertices.  The two ranges
$s\le4$ and $s\ge5$ both give a total at most $11$.  In $H_6$, the
type-$(2,1,1)$ vertices with a fixed doubled coordinate form a matching of
size at most three, reduced to two when a type-$(3,1)$ vertex with that heavy
coordinate is present.  Hence the selected vertices of types $(3,1)$ and
$(2,1,1)$ have total size at most $7\cdot3=21$.  Let $c$ be the number of
selected squarefree four-sets.  If $c\le5$, the total is at most $21+c\le26$.
If $c\ge6$, observe that the three orientations of type $(2,1,1)$ on a fixed
support triple form a clique.  Each selected squarefree four-set forbids the
four support triples that it contains, and two selected squarefree four-sets
share no triple.  At most $35-4c$ type-$(2,1,1)$ vertices can therefore be
selected.  There are at most seven type-$(3,1)$ vertices, one for each heavy
coordinate, so in this case
\[
 |I|\le 7+(35-4c)+c=42-3c\le24.
\]
Thus $\alpha(H_6)\le26$, and the cyclic fiber attains equality.

The vertex sets of $H_9$ and $H_{10}$ partition into seven heavy-coordinate
cliques,
and the same matching argument gives $\alpha(H_{11})=21$.  For $H_{12}$,
regard the selected type-$(2,2)$ vertices as the edges of a graph $K$ on
seven coordinates, with degrees $d_i$.  The remaining contribution at
coordinate $i$ is at most $3-\lfloor d_i/2\rfloor$, whence
\[
 |I|\le |E(K)|+21-\sum_i\left\lfloor\frac{d_i}{2}\right\rfloor
 =21+\frac{\#\{i:d_i\text{ odd}\}}2\le24.
\]
Cyclic fibers attain these bounds.  For $H_{13}=G_7(5)$ the checksum fiber
gives an independent $66$-set.  For the upper bound, suppose that $I$ has
size $67$, and let $n_i$ count its vertices of the ordered partition types
\[
 (5),(4,1),(3,2),(3,1,1),(2,2,1),(2,1,1,1),(1^5),
 \qquad 0\le i\le6.
\]
The orbit sizes give the coordinatewise bounds
\[
 0\le(n_0,\ldots,n_6)\le(7,42,42,105,105,140,21),
 \qquad \sum_{i=0}^6 n_i=67.
\]
The seven positive local families are, in order: the seven upward cliques
$4e_i+\Delta_7(1)$; the thirty-five upward cliques
$\mathbf 1_S+\Delta_7(1)$ with $|S|=4$; the forty-two downward two-vertex
cliques obtained from $5e_i+e_j$, $i\ne j$; the forty-two translates of
$J(5,2)$ with complementary ordered base entries $1$ and $2$; the forty-two
translates of $J(7,2)$ with base type $(2,1)$; the seven translates
$2e_j+H_9$; and the seven translates $e_j+H_{10}$.  Their respective local
capacities are $1,1,1,2,3,7,7$.  Summing all translates within these seven
families gives, in the same order, the exact aggregate inequalities
\begin{align*}
 n_0+n_1&\le7,                 &n_5+5n_6&\le35,\\
 6n_0+n_1&\le42,              &3n_5&\le84,\\
 n_2+2n_3+2n_4+3n_5&\le126,\\
 n_0+n_1+2n_2+2n_4&\le49,\\
 n_0+2n_1+n_2+2n_3&\le49.
\end{align*}
The verifier reconstructs the seven local families, their capacities, and
every displayed incidence coefficient.  Direct integer enumeration of this
displayed system gives exactly fifty type-count vectors, which are retained in
the immutable manifest.  Four vectors require four squarefree words; their
complementary two-sets would have to be a matching on seven points, an immediate
contradiction.  Each of the remaining forty-six vectors gives a
deterministically generated CNF with a retained DRAT refutation accepted by
the proof checker.  The reduction also reconstructs all $462$ vertices, so the
fifty profiles are exhaustive.  This proves $\alpha_7(5)=66$ without trusting
a SAT-solver status; see~\cite{ZabokritskiyCompanion2026}.
\end{proof}

Cap every anchor coordinate at three and let a tile weight depend only on its
type and its sorted capped anchor profile.  There are $84$ anchor profiles and
hence $13\cdot84=1092$ variables.  Vertex coordinates may be capped at eight,
leaving $3003$ saturated vertex states.

\begin{proposition}[Exact eventual certificate]
\label{prop:q7-eventual-cover}
There is a nonnegative rational state vector, supported on $435$ variables,
which covers $G_7(d)$ for every $d\ge26$ and has cost
\begin{equation}
 \frac17\binom{d+6}{6}+\delta_7,
 \qquad 0<\delta_7<1.
 \label{eq:q7-eventual-cost}
\end{equation}
\end{proposition}

\begin{proof}
The archived rational vector and verifier reconstruct all $1092$ variables,
check nonnegativity and the $3003$ saturated state rows, and check the $1547$
unsaturated profiles needed for $26\le d\le49$.  They also verify every
coefficient of the exact cost identity over $\mathbb Q$.  From degree $50$
onward a coordinate is at least eight, so the saturated verification applies.
The local-capacity inputs, including $H_{13}$, are checked by the preceding
lemma and its retained proof objects.  Thus the finite-state reduction proves
the proposition~\cite{ZabokritskiyCompanion2026}.
\end{proof}

\begin{proposition}[Exact transition range]
\label{prop:q7-transition}
The equality $\alpha_7(d)=L_7(d)$ holds for every $7\le d\le25$.
\end{proposition}

\begin{proof}
The checksum fibers give the lower bounds.  For $7\le d\le14$, matching
rational primal and dual orbit vectors certify the exact local-cover optimum.
For $15\le d\le25$, rational covers have cost strictly below $L_7(d)+1$.
The verifier regenerates every orbit incidence, checks all primal and dual
inequalities and objectives over $\mathbb Q$, and checks each strict cost
comparison.  Integrality then gives the asserted upper bounds
\cite{ZabokritskiyCompanion2026}.
\end{proof}

\begin{proposition}[The exceptional degree; computer-assisted]
\label{prop:q7-d6}
\begin{equation}
 \alpha_7(6)=133.
 \label{eq:q7-d6}
\end{equation}
\end{proposition}

\begin{proof}
Label the coordinates by the seven nonzero vectors of $\mathbb F_2^3$ and
take the zero-syndrome fiber.  A unit transfer changes its syndrome by the
sum of two distinct labels, so the fiber is independent.  Character averaging
gives its size as
\[
 \frac18\left[\binom{12}{6}+7\binom63\right]=133.
\]

For the upper bound, use the following exact rational local cover.  Upward
cliques with degree-five anchor types $(5)$, $(4,1)$, $(3,2)$, and $(2,2,1)$
receive weights $3/5$, $1/3$, $1/2$, and $1/6$, respectively.  Downward
cliques with degree-seven centers of types $(6,1)$ and $(1^7)$ receive weights
$1/15$ and $1$.  Finally, each translated $J(7,2)$ tile with degree-four
anchor of type $(3,1)$ or $(2,1,1)$ receives weight $1/6$; its local capacity
is three.  These are $386$ positive-weight tiles.  Direct incidence counting
gives coverage exactly one at every one of the $924$ vertices and total cost
$134$.

Suppose that an independent set $I$ had size $134$.  Equality in the weighted
local-cover bound would then force every positive-weight tile $T$ to be
saturated:
\[
 |I\cap T|=\alpha(T).
\]
For a partition $\lambda$ of six, write $x_\lambda$ for the number of selected
vertices of type $\lambda$.  Summing the saturated tile equalities gives the
following complete type parameterization, where
\[
 b=x_{(3,3)},\qquad c=x_{(2,1^4)},\qquad
 t=\frac{x_{(3,2,1)}}6=\frac{28-x_{(3,1,1,1)}}2.
\]
\[
\begin{array}{c|c}
\lambda&x_\lambda\\
\hline
(6)&7\\
(5,1)&0\\
(4,2)&42-2b-6t\\
(4,1,1)&b+3t\\
(3,3)&b\\
(3,2,1)&6t\\
(3,1,1,1)&28-2t\\
(2,2,2)&-7-4t+2c\\
(2,2,1,1)&63+3t-3c\\
(2,1^4)&c\\
(1^6)&1
\end{array}
\]
In particular, $3t=x_{(4,1,1)}-b$.  Thus both $2t$ and $3t$ are integers,
so $t$ is integral; nonnegativity of the displayed counts gives
$t\in\{0,\ldots,7\}$.

The unique selected squarefree word may, by coordinate symmetry, be fixed as
$(0,1,1,1,1,1,1)$.  Its $S_6$ stabilizer has exactly three orbits on the
type-$(3,1,1,1)$ layer, with representatives
\[
 (3,1,1,1,0,0,0),\qquad
 (1,3,1,1,0,0,0),\qquad
 (0,3,1,1,1,0,0).
\]
The orbit sizes are $20,60,60$, and the corresponding stabilizers are
$S_3\times S_3$, $S_2\times S_3$, and $S_3\times S_2$.  Since
$x_{(3,1,1,1)}=28-2t>0$, fixing one of these representatives gives three
lossless branches.  A deterministic exact-integer enumeration eliminates
$t=5,6,7$, visiting respectively $93$, $73$, and $39$ search nodes.  For each
of $t=0,\ldots,4$ and each of the three branches, a deterministic encoder
reconstructs the constraints from all $924$ original vertex variables and
produces one CNF, giving fifteen formulas.  Every CNF has a retained core DRAT
refutation, and the verifier regenerates each formula and invokes the proof
checker on its derivation.  Thus no discovery-solver status enters the proof,
and no independent $134$-set exists
\cite{ZabokritskiyCompanion2026}.
\end{proof}

\begin{theorem}[Seven-variable spreading numbers]
\label{thm:q7}
For every $d\ge1$,
\begin{equation}
 \alpha_7(d)=
 \begin{cases}
 7,&d=2,\\
 14,&d=3,\\
 35,&d=4,\\
 133,&d=6,\\
 \displaystyle\left\lceil\frac17\binom{d+6}{6}\right\rceil,
   &d\notin\{2,3,4,6\}.
 \end{cases}
 \label{eq:q7-main}
\end{equation}
In particular, the prime-checksum formula holds for every $d\ge7$, and the
threshold seven is sharp.
\end{theorem}

\begin{proof}
Theorem~\ref{thm:small-fixed-degree} gives $d\le4$, and
Lemma~\ref{lem:q7-local-tiles} gives the finite base case $d=5$.
Proposition~\ref{prop:q7-d6} handles degree six, while
Proposition~\ref{prop:q7-transition} handles $7\le d\le25$.  For $d\ge26$,
Proposition~\ref{prop:q7-eventual-cover} gives the upper cover and the cyclic
fiber gives the lower bound.  If $7\nmid d$, then
$\binom{d+6}{6}$ is divisible by seven and~\eqref{eq:q7-eventual-cost} has
the same integer part as $L_7(d)$.  If $7\mid d$, write
$\binom{d+6}{6}=7m+1$; the cover cost is strictly below $m+2$.  Integrality
gives $\alpha_7(d)\le L_7(d)$ in both cases.
\end{proof}

\begin{remark}[The nonuniform small-degree regime]
The eventual checksum phenomenon should not be expected to be uniform in the
alphabet size.  When $d<q$, a large proportion of the simplex lies near its
boundary, and constructions arising from label groups other than
$\mathbb Z_q$ can outperform the cyclic checksum fiber.  This already occurs
at the exceptional degrees $2$ and $4$ for $q=5$.  For $q=7$, the exact
small-degree values at $d=2,3,4$ also exceed the checksum value, and the
nonzero $\mathbb F_2^3$ labeling gives an independent set of size $133$ at
$d=6$, whereas the cyclic checksum fiber has size $132$.  These examples
suggest that further sporadic gaps may occur in the regime $d<q$, even when
the prime-checksum conjecture is eventually true for each fixed prime $q$.
The fixed-degree orbit reduction in Section~\ref{sec:fixed-degree} is designed
to study precisely this complementary regime.
\end{remark}

\section{Additive targets and the scope of the method}
\label{sec:general-framework}

The finite-state theorem is independent of the lower construction, but a sharp
application needs a target polynomial or quasi-polynomial.  Additive colorings
provide such targets and explain the residue classes that appear in the local
programs.  We record two forms that will also be used in the fixed-degree
comparison.

\subsection{Balance of the cyclic checksum}
The construction $\mathcal C_{q,d}(r)$ from
\eqref{eq:checksum-classes} provides an explicit independent set.  Its exact
class sizes are given in~\cite{KreindelEssayagZabokritskiy2026a}.  Put
\begin{equation}
 B_q(d)=\frac1q\binom{d+q-1}{q-1}
       =\frac1q|\Delta_q(d)|.
 \label{eq:checksum-average}
\end{equation}
For $m\ge1$, let
\[
 c_m(r)=\sum_{\substack{1\le k\le m\\(k,m)=1}}
          \exp\!\left(\frac{2\pi\mathrm i kr}{m}\right)
\]
be the Ramanujan sum.  Then
\begin{equation}
 N_r^{(q)}(d)
 =\frac1q\sum_{m\mid\gcd(q,d)}
   c_m(r)\binom{d/m+q/m-1}{q/m-1}.
 \label{eq:exact-checksum-balance}
\end{equation}
In particular, if $q$ is prime,
\begin{equation}
 N_r^{(q)}(d)=
 \begin{cases}
  \displaystyle\frac1q\binom{d+q-1}{q-1},
     &q\nmid d,\\[7pt]
  \displaystyle\frac1q\left[\binom{d+q-1}{q-1}+q-1\right],
     &q\mid d,\ r=0,\\[7pt]
  \displaystyle\frac1q\left[\binom{d+q-1}{q-1}-1\right],
     &q\mid d,\ r\ne0.
 \end{cases}
 \label{eq:prime-checksum-balance}
\end{equation}
Consequently,
\begin{equation}
 \max_{r\in\mathbb Z_q}N_r^{(q)}(d)=\left\lceil B_q(d)\right\rceil
 \qquad(q\text{ prime}).
 \label{eq:prime-checksum-maximum}
\end{equation}
If $p_0$ is the smallest prime divisor of $q$, then
\begin{equation}
 N_r^{(q)}(d)=B_q(d)+O_q\!\left(d^{q/p_0-1}\right)
 \qquad(d\to\infty).
 \label{eq:checksum-asymptotic-balance}
\end{equation}
The checksum construction and the one-deletion upper bound of
Kova\v{c}evi\'c and Tan~\cite{KovacevicTan2018} give
\begin{equation}
 \max_{r\in\mathbb Z_q}N_r^{(q)}(d)
 \le \alpha_q(d)
 \le \frac1q\binom{d+q}{q-1}
 =B_q(d)+\frac1q\binom{d+q-1}{q-2}.
 \label{eq:checksum-KT-sandwich}
\end{equation}
The left side describes a concrete independent set; the right side bounds all
independent sets.

\subsection{Other additive label groups}
The label group changes the lower-order terms of the target when $q$ is
composite.  Write
\[
 q=\prod_{p\mid q}q_p,\qquad q_p=p^{e_p},
 \qquad
 A_q=\bigoplus_{p\mid q}\mathbb F_p^{e_p}.
\]
For $S\subseteq\{p:p\mid q\}$, put
\[
 m_S=\prod_{p\in S}p,
 \qquad
 \gamma_S=\prod_{p\in S}(q_p-1),
\]
with $m_\varnothing=\gamma_\varnothing=1$.

\begin{proposition}[Largest additive fiber]
\label{prop:general-zero-fiber}
For the coloring of $\Delta_q(d)$ by $A_q$, the zero fiber has cardinality
\begin{equation}
 M_q(d)=\frac1q
 \sum_{\substack{S\subseteq\{p:p\mid q\}\\m_S\mid d}}
 \gamma_S
 \binom{d/m_S+q/m_S-1}{q/m_S-1}.
 \label{eq:general-zero-fiber}
\end{equation}
It is a largest color class.  In particular,
$\alpha_q(d)\ge M_q(d)$.
\end{proposition}

\begin{proof}
Apply the character filter to the generating function of nonnegative
multiplicity vectors.  A character whose nontrivial Sylow components are
indexed by $S$ has order $m_S$, and
\[
 \prod_{a\in A_q}(1-z\psi(a))^{-1}
 =(1-z^{m_S})^{-q/m_S}.
\]
There are $\gamma_S$ such characters.  Extracting the coefficient of $z^d$
gives~\eqref{eq:general-zero-fiber}.  In the zero fiber all character
contributions have positive sign, and the triangle inequality shows that no
other fiber is larger.
\end{proof}

Proposition~\ref{prop:general-zero-fiber} has only two regimes when
$q=p^m$ is a prime power:
\begin{equation}
 M_q(d)=
 \begin{cases}
 \displaystyle\frac1q\binom{d+q-1}{q-1},&p\nmid d,\\[8pt]
 \displaystyle\frac1q\left[
 \binom{d+q-1}{q-1}+(q-1)
 \binom{d/p+q/p-1}{q/p-1}\right],&p\mid d.
 \end{cases}
 \label{eq:prime-power-target}
\end{equation}
For $q=4$, these are exactly the odd- and even-degree targets in
Section~\ref{sec:q4frac}.  In general,
\begin{equation}
 M_q(d)=\frac1q\binom{d+q-1}{q-1}
       +O_q\bigl(d^{q/p_0-1}\bigr).
 \label{eq:general-target-asymptotic}
\end{equation}

Four exact-density tile families occur in the fixed-alphabet applications.  The
ternary proof uses full simplices of residual degrees
$1,5,7,8$; the quaternary proof uses two oriented cliques in odd degree and
full simplices of degrees $1,3$ in even degree; the eventual five-variable proof uses the full simplices of residual degrees
in~\eqref{eq:q5-residual-set}; and the seven-variable
proof uses the mixed orbit-union templates in Section~\ref{sec:q7}.  The examples suggest
that the main remaining difficulty is not state finiteness, which follows from
Theorem~\ref{thm:finite-state-reduction}, but template selection.

\begin{conjecture}[Finite exact-density tile families]
\label{conj:finite-tile}
For every fixed $q$, there is a finite coordinate-symmetric family of induced
templates of independence density $1/q$ such that, in every divisibility
regime of~\eqref{eq:general-zero-fiber}, the associated finite-state program
certifies
\[
 \alpha_q(d)=M_q(d)
\]
for all sufficiently large $d$, apart from finitely many exceptional degrees.
\end{conjecture}

The conjecture permits shapes other than full simplices.  The four-variable
odd-degree proof needs a reflected orientation, while the seven-variable proof
uses the mixed orbit-union templates $H_4$ and $H_6$.  Thus mixed tile families are
not merely experimental: they are already essential in a sharp
higher-dimensional application.

\section{Orbit reductions at fixed degree}
\label{sec:fixed-degree}

The previous section keeps \(q\) fixed and lets the ambient degree grow.  A
second reduction is available in the opposite direction: fix the degree \(d\)
and exploit the full permutation symmetry of the \(q\) coordinates.  In
coding language, independent sets in \(G_q(d)\) are constant \(\ell_1\)-weight
\(d\) codes of minimum \(\ell_1\)-distance four.  The exact values for
weights at most four were determined by Chen, Ma, and
Zhang~\cite{ChenMaZhang2021}.  We first recover those values as explicit
weighted-tiling certificates, then solve the degree-five orbit program and
obtain exact power-of-two families in degrees six, eight, and ten.  We finally
derive a three-term fixed-degree asymptotic expansion, together with a
conjectural hierarchy extending it.  Thus the numerical results in
Theorem~\ref{thm:small-fixed-degree} are known; the exact power-of-two families
and the subsequent symbolic LP consequences are new applications of the
weighted-tiling formalism.

The two tile families used in this section arise naturally from the deletion
geometry.  An anchor
\(a\in\Delta_q(d-1)\) is a possible one-deletion output, and
\(a+\Delta_q(1)\) is the set of all \(q\) words obtained by reinserting one
symbol.  These words are pairwise confusable, so they form a \(K_q\) and give
the canonical inequality \(|I\cap(a+\Delta_q(1))|\le1\).  Each such clique
has density \(1/q\), matching the additive lower construction.  Moreover, a vertex \(x\)
lies precisely in the cliques anchored at \(x-e_i\) for its positive
coordinates, so the coverage equations are predecessor equations and
coordinate symmetry groups them by integer partitions.  On the squarefree
layer, supports are \(k\)-subsets and a unit transfer is a one-element
exchange; the induced graph is therefore the Johnson graph \(J(q,k)\), whose
independence number is the corresponding packing number.

For integers \(v\ge k\ge t\), let \(D(v,k,t)\) be the maximum number of
\(k\)-subsets of a \(v\)-set such that no \(t\)-subset is contained in more
than one block, and set \(D(v,k,t)=0\) when \(v<k\).  Equivalently, with the
empty Johnson layer understood when \(v<k\),
\[
 D(v,k,k-1)=\alpha(J(v,k)),
\]
where \(J(v,k)\) is the Johnson graph in which two \(k\)-subsets are adjacent
when they meet in \(k-1\) points.

\begin{theorem}[The exact values in degrees two, three, and four]
\label{thm:small-fixed-degree}
For every \(q\ge2\),
\begin{align}
 \alpha_q(2)&=q,\label{eq:fixed-d2}\\
 \alpha_q(3)&=q+D(q,3,2),\label{eq:fixed-d3}\\
 \alpha_q(4)&=q+\binom q2+D(q,4,3).\label{eq:fixed-d4}
\end{align}
\end{theorem}

\begin{proof}
For degree two, use the \(q\) upward cliques
\[
 T_i=e_i+\Delta_q(1),\qquad i\in[q],
\]
with unit weights.  A corner \(2e_i\) has coverage one and a mixed vertex
\(e_i+e_j\) has coverage two.  Since every tile is a \(K_q\), the cost is
 \(q\).  The \(q\) corners form an independent lower construction.

For degree three, the \(q\) disjoint cliques
\[
 C_i=2e_i+\Delta_q(1)
\]
cover all vertices of types \((3)\) and \((2,1)\).  The remaining squarefree
vertices form a single tile isomorphic to \(J(q,3)\).  This is an integral
zero-defect tiling of total cost \(q+D(q,3,2)\).  Equality is achieved by the
\(q\) corners together with the incidence vectors of an optimal
\(2\)-\((q,3,1)\) packing.

For degree four, give weight one to each clique
\[
 3e_i+\Delta_q(1),
\]
weight \(1/2\) to each clique
\[
 2e_i+e_j+\Delta_q(1),\qquad i\ne j,
\]
and weight one to the squarefree layer, which is isomorphic to \(J(q,4)\).
The coverage by vertex type is
\[
\begin{array}{c|ccccc}
\text{type}&(4)&(3,1)&(2,2)&(2,1,1)&(1^4)\\ \hline
F&1&3/2&1&1&1.
\end{array}
\]
Thus the cost is
\[
 q+\frac12q(q-1)+D(q,4,3)
 =q+\binom q2+D(q,4,3).
\]
The matching construction consists of all corners \(4e_i\), all vectors
\(2e_i+2e_j\), and the squarefree vectors associated with an optimal
\(3\)-\((q,4,1)\) packing.
\end{proof}

The preceding proofs combine upward cliques, a Johnson layer, and
fractionally weighted overlaps.  The same symmetry gives a finite linear
program in every fixed degree.
Let \(\mathcal P_n\) be the set of integer partitions of \(n\).  If
\(\lambda\in\mathcal P_n\), write \(\ell(\lambda)\) for its number of parts
and \(m_s(\lambda)\) for the multiplicity of the part \(s\).  For \(q\ge2\), the number of exponent vectors having positive-part multiset
\(\lambda\) is
\begin{equation}
 N_q(\lambda)=|\mathcal O_q(\lambda)|
 =\frac{(q)_{\ell(\lambda)}}{\prod_{s\ge1}m_s(\lambda)!},
 \label{eq:orbit-size-partition}
\end{equation}
where \((q)_j=q(q-1)\cdots(q-j+1)\); in particular, \(N_q(\lambda)=0\) when \(\ell(\lambda)>q\), consistently with \(\mathcal O_q(\lambda)=\varnothing\).
For \(\mu\in\mathcal P_d\) and a part value \(s\) occurring in \(\mu\), let
\(\operatorname{pred}_s(\mu)\in\mathcal P_{d-1}\) be the predecessor
partition obtained by replacing one part \(s\) by \(s-1\), and deleting that
part when it becomes zero.

\begin{proposition}[The fixed-degree orbit LP]
\label{prop:fixed-degree-orbit-lp}
Assign the same weight \(w_\lambda\) to every upward clique
\(a+\Delta_q(1)\) whose anchor \(a\) has type
\(\lambda\in\mathcal P_{d-1}\).  Then the best bound obtainable from these
cliques is
\begin{equation}
\begin{aligned}
 \text{minimize}\quad&
   \sum_{\lambda\in\mathcal P_{d-1}}N_q(\lambda)w_\lambda,\\
 \text{subject to}\quad&
   \sum_{s\ge1}m_s(\mu)w_{\operatorname{pred}_s(\mu)}\ge1
   &&(\mu\in\mathcal P_d),\\
 &w_\lambda\ge0&&(\lambda\in\mathcal P_{d-1}).
\end{aligned}
\label{eq:fixed-degree-orbit-lp}
\end{equation}
For \(q\ge d\), this LP has \(p(d-1)\) variables and \(p(d)\) coverage
constraints, independent of \(q\); only its polynomial objective depends on
\(q\).
\end{proposition}

\begin{proof}
A vertex of type \(\mu\) has one predecessor for every positive coordinate.
Reducing a coordinate whose value is \(s\) produces an anchor of type
\(\operatorname{pred}_s(\mu)\), and there are \(m_s(\mu)\) such coordinates.
This gives
the coverage constraints.  Formula~\eqref{eq:orbit-size-partition} counts the
anchors in each orbit and hence gives the objective.
\end{proof}

The first degree not covered by the exact small-weight theory is five.  The
linear program in Proposition~\ref{prop:fixed-degree-orbit-lp} can nevertheless
be solved in closed form.

\begin{proposition}[The degree-five orbit LP]
\label{prop:degree-five-orbit-lp}
For every \(q\ge7\), the optimum of
\eqref{eq:fixed-degree-orbit-lp} with \(d=5\) is
\begin{equation}
 C_{q,5}^{\mathrm{orb}}
 =\frac{q(q^3+10q^2+45q+64)}{120}.
 \label{eq:d5-orbit-cost}
\end{equation}
It is attained by the anchor-type weights
\begin{equation}
 (w_4,w_{31},w_{22},w_{211},w_{1^4})
 =\left(1,\frac{11}{30},\frac{19}{30},\frac4{15},\frac15\right).
 \label{eq:d5-orbit-weights}
\end{equation}
Consequently,
\[
 \alpha_q(5)\le\left\lfloor C_{q,5}^{\mathrm{orb}}\right\rfloor
 \qquad(q\ge7).
\]
\end{proposition}

\begin{proof}
The seven coverage values correspond, in order, to the vertex types
\[
 (5),\ (4,1),\ (3,2),\ (3,1,1),\ (2,2,1),\ (2,1,1,1),\ (1^5),
\]
and are
\[
 1,\quad \frac{41}{30},\quad1,\quad1,\quad\frac76,\quad1,\quad1.
\]
Thus~\eqref{eq:d5-orbit-weights} is feasible, and substitution in the orbit
counts gives~\eqref{eq:d5-orbit-cost}.

For optimality, use the dual variables
\begin{align*}
 y_{(5)}&=q,&
 y_{(3,2)}&=\binom q2,&
 y_{(3,1,1)}&=\frac12\binom q2,\\
 y_{(2,1,1,1)}&=\frac{q(q-1)(2q-5)}{12},&&
 y_{(1^5)}&=\frac{q(q-1)(q^2-9q+16)}{120},
\end{align*}
with the other two dual variables equal to zero.  These numbers are
nonnegative for \(q\ge7\).  Direct substitution shows that all five dual
constraints are met with equality and that the dual objective is
\eqref{eq:d5-orbit-cost}.  Strong duality completes the proof.
\end{proof}

The value in Proposition~\ref{prop:degree-five-orbit-lp} is not always sharp
as a bound for the full problem.  For \(q=5\) it gives
\(27\).  Adding five translated Johnson tiles produces the exact value
without exhaustive search.
With the cheaper upward-clique weights used below, the types \((3,1,1)\) and
\((2,1,1,1)\) each have coverage \(1/2\).  Every vertex of either type lies
in exactly one tile \(Q_i\), so assigning each \(Q_i\) weight \(1/2\) fills
precisely these two deficits.  Since \(Q_i\cong J(5,3)\) has ten vertices and
independence number two, this repair has density \(1/5\).

\begin{proposition}[A degree-five certificate for five variables]
\label{prop:q5-d5-analytic}
\[
 \alpha_5(5)=26.
\]
\end{proposition}

\begin{proof}
Give the upward cliques the anchor-type weights
\[
 (w_4,w_{31},w_{22},w_{211},w_{1^4})
 =\left(1,\frac15,\frac45,\frac1{10},\frac15\right).
\]
For each \(i\in[5]\), add the tile
\[
 Q_i=2e_i+
 \left\{e_j+e_k+e_\ell:\{j,k,\ell\}\in\binom{[5]}3\right\}
 \cong J(5,3)
\]
with weight \(1/2\).  Since \(\alpha(J(5,3))=2\), the five Johnson tiles
have total cost five.  The upward-clique cost is twenty-one.  Their combined
coverage, again ordered by the seven partitions of five, is
\[
 1,\quad\frac65,\quad1,\quad1,\quad1,\quad1,\quad1.
\]
Hence \(\alpha_5(5)\le26\).  The additive coloring supplies an independent
set of size \(26\).
\end{proof}

\subsection{Exact power-of-two families in even degrees}
\label{subsec:power-two-even}

The orbit program also gives exact values in larger degrees for infinitely
many alphabet sizes.  Let $q=2^m$, label the coordinates by
$A=\mathbb F_2^m$, and use the zero fiber of the additive coloring.  By
\eqref{eq:prime-power-target}, for even $d$ this is an independent set of size
\begin{equation}
 M_q(d)=\frac1q\left[
 \binom{q+d-1}{d}+(q-1)
 \binom{q/2+d/2-1}{d/2}
 \right].
 \label{eq:power-two-even-target}
\end{equation}
For the first three even degrees beyond the known range, the upward-clique
orbit LP gives the matching upper bound.

\begin{theorem}[Exact power-of-two families]
\label{thm:power-two-even}
For every $m\ge1$, $q=2^m$, and $d\in\{6,8,10\}$,
\begin{equation}
 \alpha_q(d)=\frac1q\left[
 \binom{q+d-1}{d}+(q-1)
 \binom{q/2+d/2-1}{d/2}
 \right].
 \label{eq:power-two-even-values}
\end{equation}
\end{theorem}

\begin{proof}
The zero fiber gives the lower bound.  For degree six, order the anchor
partitions of five as
\[
 5,\ 41,\ 32,\ 311,\ 221,\ 2111,\ 1^5
\]
and assign the corresponding upward $K_q$'s the weights
\begin{equation}
 \left(1,\frac12,\frac12,\frac{19}{72},
       \frac13,\frac5{24},\frac16\right).
 \label{eq:d6-power-two-weights}
\end{equation}
For the vertex partitions
\[
 6,\ 51,\ 42,\ 411,\ 33,\ 321,\ 3111,\ 222,\ 2211,\ 21111,\ 1^6,
\]
the coverages are
\begin{equation}
 1,\ \frac32,\ 1,\ \frac{91}{72},\ 1,\ \frac{79}{72},
 1,\ 1,\ \frac{13}{12},\ 1,\ 1.
 \label{eq:d6-power-two-coverages}
\end{equation}
Thus all orbit constraints hold.  Formula~\eqref{eq:orbit-size-partition}
gives the cost
\begin{align}
 C_{q,6}
 &=q+(q)_2+\frac{43}{144}(q)_3
   +\frac5{144}(q)_4+\frac1{720}(q)_5\notag\\
 &=\frac1q\left[
   \binom{q+5}{6}+(q-1)\binom{q/2+2}{3}
   \right]
 =M_q(6).
 \label{eq:d6-power-two-cost}
\end{align}
This proves the degree-six case.

For degrees eight and ten, the exact rational orbit vectors have fifteen and
thirty coordinates, respectively.  The archived verifier reconstructs every
partition-orbit constraint and checks all coverages and the following
falling-factorial identities over $\mathbb Q$
\cite{ZabokritskiyCompanion2026}:
\begin{align}
 C_{q,8}
 ={}&q+\frac32(q)_2+\frac{99}{128}(q)_3
 +\frac{71}{384}(q)_4+\frac7{320}(q)_5\notag\\
 &+\frac7{5760}(q)_6+\frac1{40320}(q)_7,
 \label{eq:d8-power-two-cost}\\
 C_{q,10}
 ={}&q+2(q)_2+\frac{1127}{768}(q)_3
 +\frac{409}{768}(q)_4+\frac{2021}{19200}(q)_5\notag\\
 &+\frac7{600}(q)_6+\frac1{1400}(q)_7
 +\frac1{44800}(q)_8+\frac1{3628800}(q)_9.
 \label{eq:d10-power-two-cost}
\end{align}
Exact expansion shows that each expression is $M_q(d)$ from
\eqref{eq:power-two-even-target}.  The weighted local-cover bound therefore
proves the matching upper bounds.  If $q<d$, partition types with more than
$q$ parts have orbit count zero; retaining their formal coverage constraints
only strengthens the certificate.  Hence the same certificates include
$q=2,4,8$ without separate cases.
\end{proof}

For later use, if $b\in\Delta_q(d+1)$, define the downward clique
\[
 Q_b^-:=\{b-e_i:b_i>0\}\subseteq\Delta_q(d).
\]
Any two of its vertices differ by a unit transfer, so the induced graph is a
clique; it is a $K_s$ when $b$ has support size $s$.

\begin{proposition}[Two further exact instances]

\label{prop:further-power-two-instances}
The additive target is also exact at
\[
 \alpha_8(20)=111254,
 \qquad
 \alpha_{16}(12)=1088100.
\]
\end{proposition}

\begin{proof}
The zero fibers over $\mathbb F_2^3$ and $\mathbb F_2^4$ give the lower
bounds.  Exact rational orbit covers give the matching upper bounds.  For
$(q,d)=(8,20)$ the verifier reconstructs $370$ orbit variables and $434$
coverage rows for upward and downward cliques.  For $(16,12)$ it reconstructs
$56$ upward weights, one squarefree downward $K_{13}$ weight, and $77$
coverage rows.  It checks every inequality and both objective values over
$\mathbb Q$~\cite{ZabokritskiyCompanion2026}.
\end{proof}

\begin{remark}[Where the clique pattern stops]
The known case $(q,d)=(8,4)$ and Proposition~\ref{prop:further-power-two-instances}
at $(16,12)$ might suggest a continuation with $d=q-4$.  The same
certificate mechanism already breaks at $(32,28)$: an exact rational dual
certificate for the $3718$-row, $7575$-column upward--downward clique-orbit LP
has objective strictly larger than $M_{32}(28)$
\cite{ZabokritskiyCompanion2026}.  Thus the clique constraints leave a positive
partition-orbit defect and richer local templates are required.  This is a
limitation of the present tile family, not a counterexample to
additive optimality.
\end{remark}

\subsection{Comparison with the Kova\v{c}evi\'c--Tan upper bound}
The orbit LP also gives a general asymptotic statement when \(d\) is fixed
and \(q\) grows.  This reverses the fixed-alphabet limit in
\eqref{eq:checksum-KT-sandwich}; that estimate is not uniform in \(q\).
To compare bounds in the same regime, specialize Theorem~17,
Eq.~(30), of~\cite{KovacevicTan2018} to one deletion and put \(n=d\).
The insertion-side alternative in that theorem is of order \(q^d\).  Its
deletion-side alternative, with threshold \(\ell\), becomes
\begin{equation}
 U_\ell(q,d)
 =\frac1\ell\binom{q+d-2}{d-1}
  +\sum_{i=1}^{\ell-1}\binom qi\binom{d-1}{i-1}.
 \label{eq:KT-threshold-fixed-degree}
\end{equation}
For \(1\le\ell\le d-1\), the leading ratio \(U_\ell(q,d)/B_q(d)\)
is \(d/\ell+O_d(q^{-1})\), and hence is minimized at
\(\ell=d-1\).  At \(\ell=d\), the leading ratio jumps to
\(1+d(d-1)\); for \(\ell\ge d+1\), the \(i=d\) term makes
\(U_\ell(q,d)=\Theta_d(q^d)\).  Thus the best leading behavior supplied by
that theorem is
\begin{align}
 U_{d-1}(q,d)
 &=\frac1{d-1}\binom{q+d-2}{d-1}
   +\sum_{i=1}^{d-2}\binom qi\binom{d-1}{i-1}\notag\\
 &=\left(\frac d{d-1}+O_d(q^{-1})\right)B_q(d).
 \label{eq:KT-fixed-degree-comparison}
\end{align}
The support-shadow refinement in
\cite{KreindelEssayagZabokritskiy2026a} improves lower-order terms but has
the same leading ratio \(d/(d-1)\).  These bounds therefore leave a constant
relative gap.  The theorem below closes that gap, replacing the ratio by
\(1+O_d(q^{-3})\) and determining the first three powers of \(q\).
On the lower-bound side, if \(q>d\) is prime, then
\eqref{eq:prime-checksum-balance} shows that every cyclic-checksum code has
size exactly \(B_q(d)\); the new contribution here is the universal upper
bound.

Define the defect of a partition \(\lambda\in\mathcal P_n\) by
\[
 \operatorname{def}(\lambda)=n-\ell(\lambda).
\]
For \(\lambda\in\mathcal P_{d-1}\), the quantity
\(\operatorname{def}(\lambda)=d-1-\ell(\lambda)\) counts how many coordinate
collisions separate the anchor from the squarefree type.  Since
\(N_q(\lambda)=\Theta_d(q^{\ell(\lambda)})\), defect-\(r\) anchors contribute
at order \(q^{d-1-r}\).  Hence only defects zero, one, and two can affect the
first three asymptotic coefficients.  The exceptional weights are not
guessed: they are forced recursively by exact coverage of the corresponding
vertex types:
\[
\begin{aligned}
 d\,w_{1^{d-1}}&=1,\\
 w_{1^{d-1}}+(d-2)w_{2\,1^{d-3}}&=1,\\
 w_{2\,1^{d-3}}+(d-3)w_{3\,1^{d-4}}&=1,\\
 2w_{2\,1^{d-3}}+(d-4)w_{2^2 1^{d-5}}&=1.
\end{aligned}
\]
For \(d\ge7\), the proof below shows that giving weight one beyond anchor
defect two preserves feasibility and confines vertex overcoverage to types of
defect at least three.  There are \(O_d(q^{d-3})\) such vertices, and the
\(K_q\) defect identity divides this excess by \(q\), leaving only
\(O_d(q^{d-4})\) in the cost.

\begin{theorem}[Sharp fixed-degree asymptotics through three terms]
\label{thm:fixed-degree-asymptotic}
Fix \(d\ge7\).  There is a constant \(C_d>0\) such that, for all sufficiently
large \(q\),
\begin{equation}
 0\le
 \alpha_q(d)-\frac1q\binom{q+d-1}{d}
 \le C_d q^{d-4}.
 \label{eq:fixed-degree-asymptotic}
\end{equation}
Equivalently,
\begin{align}
 \alpha_q(d)
 ={}&\frac{q^{d-1}}{d!}
 +\frac{q^{d-2}}{2(d-2)!}\notag\\
 &+\frac{3d^2-7d+2}{24(d-2)!}\,q^{d-3}
 +O_d(q^{d-4}).
 \label{eq:fixed-degree-expansion}
\end{align}
The upper bound is certified by giving every upward clique
\(a+\Delta_q(1)\) a weight determined by the type of its anchor:
\begin{equation}
 w_\lambda=
 \begin{cases}
 \displaystyle\frac1d,&\lambda=1^{d-1},\\[5pt]
 \displaystyle\frac{d-1}{d(d-2)},&\lambda=2\,1^{d-3},\\[7pt]
 \displaystyle\frac{d^2-3d+1}{d(d-2)(d-3)},
   &\lambda=3\,1^{d-4},\\[7pt]
 \displaystyle\frac{d^2-4d+2}{d(d-2)(d-4)},
   &\lambda=2^2 1^{d-5},\\[7pt]
 1,&\text{otherwise}.
 \end{cases}
 \label{eq:defect-two-weights}
\end{equation}
These weights form a fractional tile cover.
\end{theorem}

\begin{proof}
For \(d\ge7\), all four exceptional weights in
\eqref{eq:defect-two-weights} satisfy \(0<w_\lambda\le1\).  The coverage is exactly one for vertex types
\[
 1^d,\qquad 2\,1^{d-2},\qquad 3\,1^{d-3},\qquad
 2^2 1^{d-4};
\]
this follows by substituting the predecessor multiplicities in
\eqref{eq:fixed-degree-orbit-lp}.  If a vertex has larger defect and contains
a coordinate equal to one, removing that coordinate gives an anchor of defect
at least three and hence weight one.  If it has no coordinate equal to one,
then for \(d\ge7\) some predecessor also has defect at least three; again its
weight is one.  Thus every vertex has coverage at least one.

Every upward clique has \(q\) vertices, so
\[
 \sum_a w_a-\frac1q|\Delta_q(d)|
 =\frac1q\sum_{x\in\Delta_q(d)}(F(x)-1).
\]
Overcoverage occurs only on vertices of defect at least three, hence on
vectors with support at most \(d-3\).  For fixed \(d\), there are
\(O_d(q^{d-3})\) such vertices and their coverage is uniformly bounded by
\(d\).  The upper error is therefore \(O_d(q^{d-4})\).  The cyclic checksum
classes partition \(\Delta_q(d)\), so their largest member gives the matching
lower bound \(|\Delta_q(d)|/q\); for prime \(q>d\), exact equality for every
class follows from \eqref{eq:prime-checksum-balance}.  Expanding
\(q^{-1}\binom{q+d-1}{d}\) gives~\eqref{eq:fixed-degree-expansion}.
\end{proof}

\begin{remark}[The asymptotic improvement and the lower benchmark]
Equation~\eqref{eq:KT-fixed-degree-comparison} exceeds \(B_q(d)\) by
\[
 \left(\frac1{d-1}+O_d(q^{-1})\right)B_q(d)
 =\Theta_d(q^{d-1}),
\]
whereas Theorem~\ref{thm:fixed-degree-asymptotic} reduces the universal
upper-bound gap to \(O_d(q^{d-4})\).  Thus the improvement is not merely in a
lower-order constant: it closes the previous leading-coefficient gap.
For every fixed \(d\ge7\) and every sufficiently large prime \(q>d\), combining
\eqref{eq:prime-checksum-balance} with
Theorem~\ref{thm:fixed-degree-asymptotic} gives, for every
\(r\in\mathbb Z_q\),
\[
 N_r^{(q)}(d)=B_q(d),
 \qquad
  0\le \alpha_q(d)-N_r^{(q)}(d)\le C_d q^{d-4}.
\]
The exact equality is the balance result for the explicit checksum code; the
error estimate is the universal upper bound proved by the weighted tiling.
\end{remark}

The proved weight pattern above suggests the second level of a natural
hierarchy.  At a proposed level \(k\), one solves recursively for anchor types
of defect at most \(k\) so that every vertex type of defect at most \(k\) has
coverage exactly one, and gives weight one to all remaining anchor types.
Only the level realized in Theorem~\ref{thm:fixed-degree-asymptotic} is proved
here; the higher levels are conjectural.

\begin{conjecture}[The fixed-degree defect hierarchy]
\label{conj:defect-hierarchy}
For every \(k\ge0\) and every \(d\ge2k+3\), the recursive orbit weights of
levels at most \(k\) are nonnegative and form a fractional cover.  If so, then
\[
 \alpha_q(d)
 =\frac1q\binom{q+d-1}{d}
  +O_d(q^{d-k-2})
 \qquad(q\to\infty).
\]
\end{conjecture}

The exact formulas in Theorem~\ref{thm:small-fixed-degree} belong to the
existing constant-weight \(\ell_1\)-code theory.  The closed degree-five orbit
solution and the proved three-term fixed-degree expansion appear to be new
consequences of the present weighted-tiling formalism.  The higher defect
hierarchy is conjectural.

\section{Concluding remarks and the multiset interpretation}
\label{sec:coding}

A vector $a\in\Delta_q(d)$ is the multiplicity vector of a multiset of
cardinality $d$ on a $q$-symbol alphabet.  Two such multisets can yield the
same output after one deletion from each precisely when their multiplicity
vectors differ by a unit transfer.  Hence a subset of $\Delta_q(d)$ corrects
one multiset deletion if and only if it is independent in $G_q(d)$.
The cyclic fibers $\mathcal C_{q,d}(r)$ in
\eqref{eq:checksum-classes} have a direct syndrome decoder.  If
$y=a-e_j$ is received from a codeword of syndrome $r$, then
\[
 r-\sigma_q(y)=j-1\pmod q,
\]
which identifies the deleted symbol.  The exact results for three, four, and
five variables give optimal one-deletion multiset-code sizes for those
alphabet sizes, and Theorem~\ref{thm:q7} gives the optimal size for alphabet
size seven in every cardinality.  For the broader coding framework,
see \cite{KovacevicVukobratovic2013,KovacevicTan2018,
KreindelEssayagZabokritskiy2026a,KreindelEssayagZabokritskiy2026b}.

The main structural point is that two apparently different symmetry reductions
arise from the same weighted local-cover inequality.  When $q$ is fixed,
translated templates and capped boundary data give a finite-state program
valid for infinitely many degrees.  In seven variables, thirteen local
capacities and capped anchor profiles reduce all large degrees to one exact
rational certificate; a separate symmetry-reduced integer argument settles
the exceptional degree six.  When $d$ is fixed, upward cliques and coordinate
orbits give a partition-indexed program whose size is independent of $q$.

Theorems~\ref{thm:q3frac}, \ref{thm:q4frac}, \ref{thm:q5}, and \ref{thm:q7}
give new proofs of the known three- and four-symbol formulas and determine the
five- and seven-symbol cases in every degree.  For seven symbols the checksum
formula is valid from degree seven onward.  When the degree is fixed and the
alphabet grows, the results recover the known values through
degree four, solve the degree-five orbit program for $q\ge7$, determine exact
values in degrees six, eight, and ten for every power-of-two alphabet size,
and give a three-term asymptotically sharp upper bound for every fixed
$d\ge7$.

\paragraph{Open problems.}
Three directions remain especially concrete.  First, extend the
prime-checksum theorem to $q=11$.  Second, replace the successful finite
template families by a structural selection rule and prove the higher
fixed-degree defect hierarchy in Conjecture~\ref{conj:defect-hierarchy}.
Third, find richer exact-density templates
that extend Theorem~\ref{thm:power-two-even} beyond degree ten and overcome
the clique-orbit obstruction at $(q,d)=(32,28)$.
More generally, it remains open to find explicit tile families that solve
further infinite classes of parameters, as predicted by
Conjecture~\ref{conj:finite-tile}.  A stronger goal is a general theorem for
these families of finite-state linear programs: one would like checkable
conditions on a collection of local tiles that guarantee the required cover
for every sufficiently large degree, without having to find a separate
certificate for each new value of $q$.

\section*{Data and code availability}

The software, machine-readable exact certificates, deterministic generators,
retained RUP and DRAT proof objects, proof-checking tools, outputs, checksums,
and documentation supporting every computer-assisted statement are openly
available in version 1.1.0 of the computational companion
\cite{ZabokritskiyCompanion2026}.  The record includes a theorem-to-artifact
map and instructions for reproducing all exact checks; no private computation
is needed for the certified claims.

\end{document}